\documentclass[a4paper, 11pt]{article}
\usepackage [a4paper,left=2.5cm,bottom=2.5cm,right=2.5cm,top=2.5cm]{geometry}

\usepackage[english]{babel}

\usepackage[utf8]{inputenc} 	
\usepackage[T1]{fontenc}
\usepackage{lmodern}
\usepackage{yhmath}
\usepackage{comment}
\usepackage{bm}
\usepackage{enumerate}   

\usepackage{amsthm,amsmath,amssymb,amsbsy,bbm,mathrsfs,supertabular,
eurosym,graphicx,enumitem,xcolor, dsfont}
\usepackage{mathtools}
\usepackage{ulem}

\usepackage{tikz,tkz-tab} 
\usepackage{subcaption}

\newtheoremstyle{BBstyle0}  {}{}{\itshape}{}{\bfseries}{}{6pt}{}
\newtheoremstyle{BBstyle1}  {3pt}{3pt}{\rmfamily}{}{\itshape}{: }{3pt}{}
\newtheoremstyle{BBstyle2}  {3pt}{3pt}{\itshape}{}{\bfseries\large}{}{0pt}{}
\newtheoremstyle{BBstyle3}  {}{}{\itshape}{}{\bfseries}{: }{3pt}{}
\newtheoremstyle{BBstyle4}  {}{}{\rmfamily}{}{\bfseries}{}{6pt}{}
  
\newtheorem{thm}{Theorem}
\newtheorem{lemma}{Lemma}
\newtheorem{prop}{Proposition}
\newtheorem{df}{Definition}

\newtheorem{ass}{Assumption}

\newtheorem{rem}{Remark}

\theoremstyle{definition}

\usepackage[english]{babel}

\usepackage{enumitem}
\usepackage{hyperref}  

\usepackage[textwidth=1.8cm, textsize=tiny]{todonotes}

\newcommand{\norm}[1]{\left\|{#1}\right\|}

\newcommand{\R}{{\mathbb{R}}}

\newcommand{\modar}{\color{black}}

\newcommand{\indi}[1]{\mathds{1}_{#1}}
\DeclarePairedDelimiter{\abs}{\lvert}{\rvert}

\title{On the role of symmetry for staircase mechanisms in local differential privacy efficiency across different privacy regimes}
\author{Chiara Amorino\thanks{ Universitat Pompeu Fabra and Barcelona School of Economics, Department of Economics and Business, Ram\'on Trias Fargas 25-27, 08005, Barcelona, Spain. The author gratefully acknowledges financial support of PID2022-138268NB-I00/AEI/10.13039/501100011033.} \qquad Arnaud Gloter \thanks{Laboratoire de Math\'ematiques et Mod\'elisation d'Evry, CNRS, Univ Evry, Universit\'e Paris-Saclay, 91037, Evry, France.}} 

\begin{document}
\maketitle




\begin{abstract}
\noindent 
We investigate the structural foundations of statistical efficiency under $\alpha$-local differential privacy, with a focus on maximizing Fisher information. Building on the role of continuous staircase mechanisms, we identify a fundamental symmetry regarding the extremal values $1$ and $e^{\alpha}$. We demonstrate that when the optimal measure satisfies this symmetry, the Fisher information admits a closed-form expression. More generally, we derive a decomposition of the Fisher information into symmetric and asymmetric components, scaling as $\alpha^{2}$ and $\alpha^{3}$, respectively,  for $\alpha \to 0$. This reveals that, if in the high-privacy regime asymmetry is negligible, it is no longer the case as privacy constraints are relaxed.

Motivated by this, we introduce a class of fully asymmetric privacy mechanisms constructed via pushforward mappings, proving that—unlike their symmetric counterparts—they recover the full Fisher information of the non-private model as $\alpha \to \infty$. We bridge the gap between theory and practice by providing a tractable implementation of these mechanisms, governed by a tuning parameter $c$. This parameter allows for a smooth interpolation between the symmetric regime and the fully asymmetric regime. Furthermore, we demonstrate the versatility of this framework by showing that it encompasses the binomial mechanism as a limiting case.
\\
\\
\noindent
 \textit{Keywords: local differential privacy, Fisher information, efficiency, staircase mechanism} \\ 
 
\noindent
\textit{MSC2020 subject classifications: Primary 62F12, 62F30; secondary 62B15}
\end{abstract}

\tableofcontents

\section{Introduction}{\label{s: intro}

We have entered the era of big data, in which statisticians and data scientists have access to increasingly large and rich datasets. The growing availability of data offers unprecedented opportunities for statistical analysis and inference. However, a substantial portion of these data are sensitive in nature, making it essential to ensure appropriate levels of privacy protection.

Several approaches to data privacy have been proposed in the literature. Among them, differential privacy has emerged as a fundamental and widely adopted framework. The notion of global differential privacy, also referred to as central differential privacy, was first introduced in the seminal work of Dwork et al.~\cite{Dwo06}, where a trusted curator has access to the entire dataset and releases privatized outputs. Local differential privacy, by contrast, does not rely on the existence of a trusted aggregator: each user privatizes their data locally before transmission. This stronger privacy model has been adopted in practice by major technology companies such as Apple~\cite{5LDP,49LDP} and Google~\cite{1LDP,23LDP}, highlighting the relevance of local differential privacy both from a theoretical perspective and for real-world applications.

While local differential privacy provides stronger privacy guarantees, it also entails a significant loss in statistical utility. This fundamental tension has motivated a growing body of work devoted to statistical inference under local differential privacy constraints. Early and influential contributions in this direction include \cite{duchi2018minimax} and \cite{wasserman2010statistical}. Within this line of research, particular attention has been given to parametric inference; see, for instance, \cite{amorino2025evolving, asi2022optimal, duchi2019lower, joseph2019locally, kalinin2024efficient}. A central theme is moreover the construction of statistically efficient estimators under privacy constraints.

Although a rigorous definition of local differential privacy is deferred to later in the paper (see \eqref{eq: def LDP1}), we highlight here that the level of privacy is controlled by a parameter \( \alpha \in [0,\infty) \). The case \( \alpha = 0 \) corresponds to perfect privacy, whereas privacy constraints become progressively weaker as \( \alpha \to \infty \). As expected, stronger privacy guarantees (i.e., smaller values of \( \alpha \)) lead to a reduction in statistical information, limiting the extent to which the statistician can exploit the data. This inherent privacy--utility trade-off has motivated extensive research aimed at characterizing optimal privacy levels and, for a fixed privacy budget \( \alpha \), at identifying privacy mechanisms that achieve the best possible statistical performance, according to a particular utility function that one aims at maximizing/minimizing.

In this paper, we focus on the latter problem. Specifically, for a fixed level of privacy \( \alpha \), we investigate the design of privacy mechanisms that are optimal in the sense of maximizing the Fisher information. This optimality criterion has been extensively studied in the literature and plays a central role in our analysis.

In the context of central differential privacy, the early works \cite{Smi08,Smi11} initiated the study of efficient parameter estimation under privacy constraints. In the local differential privacy setting, \cite{BarChe20} derived upper bounds on the Fisher information achievable under LDP, though without constructing explicit privacy mechanisms attaining these bounds. Subsequent progress was made by \cite{NamLee22}, who investigated Fisher information maximization under one-bit communication schemes.

More recently, several works have further advanced the understanding of efficiency under LDP, including \cite{duchi2024right, steinberger2024efficiency,amorino2025factorization}. In particular, \cite{duchi2024right} introduced the notion of \(L_1\)-information to characterize the local minimax risk of regular parametric models under LDP, up to universal constants. Recently, \cite{steinberger2024efficiency} established the existence of asymptotically efficient locally private estimators and provided explicit constructions achieving minimum variance by solving a continuous optimization problem via a discretization approach. Finally, \cite{amorino2025factorization} derived an exact continuous solution to the same optimization problem, thereby eliminating the need for any discretization step.

Across these recent contributions, it has become clear that staircase mechanisms—also referred to as extremal mechanisms—play a central role in optimization problems under privacy constraints (see Definition~\ref{def: staircase matrix} for a precise definition). The one-dimensional staircase mechanism was introduced by \cite{GenVis12}, who showed both theoretically and empirically that its piecewise-constant noise outperforms the Laplace mechanism. This construction was later extended to higher dimensions in \cite{Gen15}, where the authors again demonstrated superiority over Gaussian and Laplace noise for \( d \geq 2 \). More recently, \cite{Kul23} proposed a staircase-like scheme for one-dimensional mean estimation, while \cite{kairouz2016extremal} formulated the privacy--utility trade-off as a constrained optimization problem and proved that, on finite spaces, the class of staircase mechanisms contains the exact optima for a broad family of information-theoretic utility criteria. Staircase mechanisms also play a crucial role in \cite{steinberger2024efficiency} and \cite{kalinin2024efficient}, where they are used to explicitly construct optimal privacy mechanisms that attain the lower bounds derived in the analysis, as well as in \cite{yoon2025fundamental}, which is devoted to discrete distribution estimation under utility-optimized local differential privacy. Moreover, related ideas have been employed in \cite{butucea2025sample} to construct gentle measurements, exploiting the connection between gentleness and differential privacy.

To address non-discrete domains, \cite{amorino2025factorization} introduced the notion of continuous staircase mechanisms and established their effectiveness for Fisher information maximization under LDP in infinite-dimensional settings. In that work, both upper and lower bounds on the Fisher information were derived in the continuous case, and it was shown that these bounds coincide as \( \alpha \to 0 \). Moreover, for fixed \( \alpha \), the maximization problem was shown to admit a solution: specifically, there exists a Radon sub-probability measure \( \bar{\mu} \) on the set of extremal points \( \mathcal{E} \) (see \eqref{eq: set E}) such that the associated continuous staircase mechanism \( q^{(\bar{\mu})} \) solves the Fisher information maximization problem.
{\modar The support of $\bar{\mu}$ is the set of labels where the privacy mechanism outputs the public variables.} Nevertheless, little is currently known about the structure of the measure \( \bar{\mu} \), or about the explicit form of the resulting privacy mechanism and its Fisher information.

This work builds on the observation that the structure of the set of extremal values—most notably the values \(1\) and \(e^{\alpha}\)—plays a fundamental role in determining the behavior of continuous staircase mechanisms.  This insight motivates our focus on measures that are invariant under swapping these extremal values. {\modar It is noticeable that the binomial privacy mechanisms, which are optimal for $\alpha\to0$, are associated to measures satisfying this symmetry property.}
{\modar We demonstrate that a symmetrization procedure can be used to construct a randomization mechanism from any positive Radon measure on the set of extremal points $\mathcal{E}$. Within this framework, binomial mechanisms emerge as the special case where the underlying measure is a Dirac mass centered at a point in $\mathcal{E}$.}
We show that, when the Radon sub-probability measure \( \bar{\mu} \) solving the maximization problem satisfies this symmetry property, the associated Fisher information admits a closed-form expression for any fixed \( \alpha \), {\modar and the mechanism reduces to a binomial mechanism}. This observation naturally leads us to study this form of symmetry in greater depth. In doing so, we uncover that even when the measure solving the optimization problem is not symmetric, the corresponding Fisher information can still be decomposed into a symmetric and an asymmetric component. As to be expected, the asymmetric component vanishes when the measure is symmetric, thus recovering the previous result. More interestingly we show that, in the asymptotic regime $\alpha \to 0$, the symmetric component scales as \( \alpha^2 \), whereas the asymmetric one is of order \( \alpha^3 \). As a consequence, in the high-privacy regime, the contribution of the asymmetric part becomes negligible. This phenomenon no longer holds for larger values of \( \alpha \), which provides one motivation for studying asymmetric privacy mechanisms in detail. Another motivation originates from practical considerations: 
{\modar general $\mathcal{E}$-valued privacy}
mechanisms  can be difficult to implement in practice. With this in mind, in Section~\ref{ss:asymm} we analyze a class of asymmetric privacy mechanisms whose associated measures on \( \mathcal{E} \) are obtained via a pushforward construction from probability measures $\nu$ defined on the original alphabet \( \mathcal{X} \). In this way we are able to implement a large class of privacy mechanisms indexed by $\nu$. Notably, we show that, for such mechanisms, taking the limit \( \alpha \to \infty \) of the associated Fisher information recovers the same amount of information as in the absence of privacy constraints. This is in sharp contrast with the behavior of symmetric two-point mechanisms, for which such a recovery does not occur. To bridge the gap between these theoretical gains and practical utility, we provide a concrete, implementable algorithm for these mechanisms, introducing a tuning parameter \(c\) that continuously interpolates between the symmetric (\(c=1/2\)) and fully asymmetric regimes. Numerical experiments on Gaussian translation models confirm that while symmetry is optimal for strong privacy, breaking symmetry by tuning \(c\) 
reduces variance in low-privacy settings. Finally, we demonstrate the versatility of our framework by constructing a specific asymmetric approximation that converges to the binomial mechanism. This allows our continuous framework to seamlessly replicate the asymptotic optimality of discrete mechanisms as \(\alpha \to 0\), effectively offering a unified approach that remains efficient across the entire privacy spectrum. \\
\\
The paper is organized as follows. Section \ref{s: pb formulation} establishes the mathematical framework, introducing formal definitions for both discrete and continuous extremal mechanisms. In Section \ref{s: main}, we present our main contributions, beginning with a detailed analysis of symmetric privacy mechanisms and their optimality in the high-privacy regime. This naturally leads to the study of asymmetric privacy mechanisms in Section \ref{ss:asymm}, where we demonstrate their advantages as privacy constraints are relaxed. We then turn to practical applications in Sections \ref{ss: Implementation} and \ref{ss:approx}, providing a concrete implementation of our method and developing an asymmetric approximation of the binomial mechanism. Finally, Section \ref{s: proof main} contains the proofs of our main results.

\section{Problem formulation}{\label{s: pb formulation}}

As already mentioned, this paper focuses on efficiency under $\alpha$-LDP constraints and aims to investigate the role of symmetry in this context. The first step is to formally define local differential privacy. The process of transforming raw data into public data is modeled through a conditional distribution, referred to as a privacy mechanism or channel distribution.  

Specifically, let $\mathcal{X}$ and $\mathcal{Z}$ be two separable, complete metric spaces. Equipped with their respective Borel $\sigma$-algebras, they form the measurable spaces $(\mathcal{X}, \Sigma_X)$ and $(\mathcal{Z}, \Sigma_Z)$. The first represents the space of sensitive data to which the privacy mechanism is applied. The resulting output $Z$ lives in $(\mathcal{Z}, \Sigma_Z)$, corresponding to the space of privatized (public) data.  

To formalize the transformation of raw samples into public views, let $X_1, \dots, X_n$ be i.i.d.\ private observations taking values in a domain $\mathcal{X}$.
While privacy mechanisms generally fall into two categories—interactive and non-interactive—this work focuses primarily on the non-interactive setting.
However, for the sake of completeness, we first recall the definition of a sequentially interactive mechanism. In this broader context, the privatization of the $i$-th observation allows for dependency not only on the current private datum $X_i$, but also on the history of previously released public outputs $Z_1, \dots, Z_{i-1}$. This results in the following conditional independence structure:
\[
X_i, Z_1, \dots, Z_{i-1} \rightarrow Z_i, \qquad Z_i \perp X_k \mid X_i, Z_1, \dots, Z_{i-1} \quad \text{for } k \neq i.
\]
Hence, the privatized output $Z_i$ is drawn according to
\[
Z_i \sim Q_i(\cdot \mid X_i = x_i, Z_1 = z_1, \dots, Z_{i-1} = z_{i-1}),
\]
for a collection of Markov kernels $Q_i : \Sigma_Z \times \mathcal{X} \times \mathcal{Z}^{i-1} \rightarrow [0,1]$.  

In contrast, in the non-interactive setting, the public variable $Z_i$ depends solely on the corresponding raw value $X_i$ and is independent of the previously generated variables $Z_1, \dots, Z_{i-1}$. Consequently, the Markov kernel $Q$ does not depend on $i$, and the same mechanism is applied to all observations. Therefore, for each $i \in \{1, \dots, n\}$, the privatized output is drawn as
\begin{equation}\label{eq: drawn Z nonint}
    Z_i \sim Q(\cdot \mid X_i = x_i).
\end{equation}
Observe that both interactive and non-interactive mechanisms have been widely studied in the literature. On one hand, it is intuitive that working with non-interactive mechanisms is simpler, since starting from independent raw data results in independent public data, which is not the case for interactive mechanisms. On the other hand, the extra dependence in the interactive setting is due to the fact that each $Z_i$ carries not only information about $X_i$ but also about $X_1, \dots, X_{i-1}$, which can enhance statistical inference.

What is important in both cases is that the concept of local differential privacy allows us to quantify the amount of privacy injected into the system, through the parameter~$\alpha$. This naturally leads to the definition of $\alpha$-local differential privacy. Given a parameter $\alpha \ge 0$, we say that a random variable $Z_i$ is an $\alpha$-locally differentially private version of $X_i$ if, for all $z_1, \dots, z_{i-1} \in \mathcal{Z}$ and all $x,x' \in \mathcal{X}$, the following condition holds:
\begin{equation}{\label{eq: def LDP1}}
\sup_{A \in \Sigma_Z} \frac{Q_i\bigl(A \mid X_i = x, Z_1 = z_1, \dots, Z_{i-1} = z_{i-1}\bigr)}
{Q_i\bigl(A \mid X_i = x', Z_1 = z_1, \dots, Z_{i-1} = z_{i-1}\bigr)}
\le \exp(\alpha).    
\end{equation}

Recall that we have denoted by $\mathcal{Q}_\alpha$ the set of all Markov kernels satisfying this local $\alpha$-differential privacy constraint. This set will play a central role, as our optimization problem will involve searching for the privacy mechanism within $\mathcal{Q}_\alpha$ that maximizes the Fisher information.  

In the non-interactive case, the definition of $\alpha$-LDP simplifies to
\[
\sup_{A \in \Sigma_Z} \frac{Q(A \mid X = x)}{Q(A \mid X = x')} \le \exp(\alpha).
\]
We underline that the parameter $\alpha$ quantifies the strength of privacy: the smaller the value, the harder it is to infer sensitive information from the released data. In particular, $\alpha = 0$ corresponds to perfect privacy, while letting $\alpha \to \infty$ gradually relaxes the privacy constraint.  

Under the local differential privacy condition, the kernels $Q(\cdot \mid X = x)$ are mutually absolutely continuous for all $x \in \mathcal{X}$. Therefore, we may assume the existence of a dominating measure~$\mu$ on $(\mathcal{Z}, \Sigma_Z)$ with respect to which each kernel admits a density, denoted $q(x,z)$. The $\alpha$-local differential privacy constraint can then be rewritten in terms of these densities, giving rise to the following equivalent formulation, which will be the one we use most frequently in the sequel:
\begin{equation}\label{eq: def LDP}
\sup_{z \in \mathcal{Z}} \frac{q(x,z)}{q(x',z)} \le \exp(\alpha), \quad \forall x, x' \in \mathcal{X}.
\end{equation}
In what follows, we will write $q(x,z)$ to denote the density $q(z \mid X = x)$.  

Finally, note that if $q(x,z) = 0$ for some $x \in \mathcal{X}$ and $z \in \mathcal{Z}$, then the $\alpha$-LDP constraint implies that $q(x',z) = 0$ for all $x' \in \mathcal{X}$. In other words, such a value $z$ could be removed from $\mathcal{Z}$ without altering the randomization mechanism.

\subsection{The staircase mechanism}

Having established the necessary background on data privatization, we now turn to the problem of efficient parameter estimation under $\alpha$-LDP. This problem has been widely investigated in the literature, as discussed in the introduction. In this context, staircase mechanisms play a central role. Before proceeding with our analysis, we recall their definition in the finite-alphabet setting.  

A key reference on this topic is \cite{kairouz2016extremal}, which shows that when $\mathcal{X}$ is finite, the optimal mechanism maximizing utility functions expressible as sums of sublinear functions has an extremal structure. We will then briefly mention their extension to infinite-dimensional settings, as recently proposed in \cite{amorino2025factorization}.  

Following \cite{kairouz2016extremal}, a randomization mechanism is called extremal if, for all $z \in \mathcal{Z}$ and all $(x, x') \in \mathcal{X}^2$, the log-likelihood ratios satisfy
\[
\left| \ln \frac{q(x',z)}{q(x,z)} \right| \in \{0, \alpha\},
\]
which is equivalent to the condition
\begin{equation}\label{eq : extremal constrainte discret}
\forall z \in \mathcal{Z}, \ \forall (x,x') \in \mathcal{X}^2, \quad \frac{q(x',z)}{q(x,z)} \in \{e^{-\alpha}, 1, e^\alpha\}.
\end{equation}
Any mechanism satisfying this constraint is referred to as a staircase mechanism. One of the main reasons staircase mechanisms have been so intensively studied is that they saturate the $\alpha$-LDP constraint, which naturally places them in a strong position for proving optimality in constrained optimization problems. In particular, the so-called staircase pattern matrix, as defined below, appears naturally in such problems (see, e.g., Theorem 4 in \cite{kairouz2016extremal} or Section 2.1.3 of \cite{kalinin2024efficient}).

\begin{df}\label{def: staircase matrix}
For $\beta \in \{0, \dots, 2^d - 1\}$, let $d_j(\beta)$ denote the $j$-th component of the $d$-dimensional binary vector corresponding to the dyadic expansion of $\beta$, that is,
\[
\beta = \sum_{j=0}^{d-1} d_j(\beta) 2^j, \quad \text{with } d_j(\beta) \in \{0, 1\}.
\]
A matrix is called a staircase pattern matrix if its $\beta$-th column is a vector $r_\beta \in \{1, e^\alpha\}^d$ for each $\beta \in \{0, \dots, 2^d - 1\}$, where the $j$-th component of $r_\beta$ is defined by
\[
(r_\beta)_j = \mathds{1}_{\{d_j(\beta) = 0\}} + e^\alpha \, \mathds{1}_{\{d_j(\beta) = 1\}}.
\]
Each such column vector $r_\beta$ is called a staircase pattern.
\end{df}

To clarify this construction, consider the case $d = 3$. There are $2^d = 8$ staircase patterns, and the corresponding staircase pattern matrix is
\[
\begin{bmatrix}
1 & 1 & 1 & 1 & e^\alpha & e^\alpha & e^\alpha & e^\alpha \\
1 & 1 & e^\alpha & e^\alpha & 1 & 1 & e^\alpha & e^\alpha \\
1 & e^\alpha & 1 & e^\alpha & 1 & e^\alpha & 1 & e^\alpha
\end{bmatrix}.
\]
This structure can be described for any $d$ in a general way by introducing the following vectors $r$
\begin{equation}\label{eq: R_matrix_by_slice}
    r_0 = \begin{bmatrix} 1 \\ 1 \\ 1 \\ \vdots \\ 1 \end{bmatrix}, \quad
    r_1 = \begin{bmatrix} e^\alpha \\ 1 \\ 1 \\ \vdots \\ 1 \end{bmatrix}, \quad
    r_2 = \begin{bmatrix} 1 \\ e^\alpha \\ 1 \\ \vdots \\ 1 \end{bmatrix}, \quad \dots \quad
    r_{2^d - 1} = \begin{bmatrix} e^\alpha \\ e^\alpha \\ e^\alpha \\ \vdots \\ e^\alpha \end{bmatrix}.
\end{equation}
One sees that the vectors $\bigl(r_\beta\bigr)_{\beta \in \{0, \dots, 2^d - 1\}}$ correspond to the extremal points of the hyper-rectangle $[1, e^\alpha]^d$. In the finite-dimensional case, one can define $\mathcal{E} = \{0, \dots, 2^d - 1\}$ as the index set of these extremal patterns $\{r_0, \dots, r_{2^d - 1}\}$.  

In particular, if $\mathcal{X}$ has finite cardinality, say $\mathcal{X} = \{ x_1, \dots, x_d \}$, each $r_\beta$ can be interpreted as a function on $\mathcal{X}$, with
\[
r_\beta(x_j) = (r_\beta)_j, \quad \text{for } x_j \in \mathcal{X}.
\]
For each $\beta \in \mathcal{E}$, the level sets
\begin{equation}\label{eq: Fbeta+}
    F^+_\beta = \{ x \in \mathcal{X} \mid r_\beta(x) = e^\alpha \}, \quad
    F^-_\beta = \{ x \in \mathcal{X} \mid r_\beta(x) = 1 \}
\end{equation}
play a crucial role, so that for all $x \in \mathcal{X}$, the function $r_\beta$ satisfies
\begin{equation}\label{eq: r_beta_x}
    r_\beta(x) = e^\alpha \mathds{1}_{F^+_\beta}(x) + \mathds{1}_{F^-_\beta}(x)
    = 1 + \bigl(e^\alpha - 1\bigr) \mathds{1}_{F^+_\beta}(x),
\end{equation}
i.e., the function $r_\beta$ takes values in $\{1, e^\alpha\}$ with its structure fully determined by $F^+_\beta$.  

Since here we focus on the setting where $\mathcal{X}$ may be continuous, we rely on the extension of staircase mechanisms introduced in Section 3 of \cite{amorino2025factorization}. In this case, the hyper-rectangle $[1, e^\alpha]^d$ of the finite setting is replaced by the set
\begin{equation}{\label{eq: set B}}
B := \left\{ v : \mathcal{X} \to \mathbb{R} \ \text{measurable} \;\bigg|\; 1 \le v(x) \le e^\alpha \ \text{a.e.} \right\} \subset L^\infty(\mathcal{X}, \mathbb{R}),
\end{equation}
and the extremal points of $[1, e^\alpha]^d$ are replaced by the extremal points of $B$, namely
\begin{equation}{\label{eq: set E}}
\mathcal{E} := \left\{ r : \mathcal{X} \to \mathbb{R} \ \text{measurable} \;\bigg|\; r(x) \in \{1, e^\alpha\} \ \text{a.e.} \right\}.
\end{equation}
Analogously to the level sets in \eqref{eq: Fbeta+} and the associated vectors in \eqref{eq: r_beta_x}, for any $r \in \mathcal{E}$ we define
\begin{equation}{\label{eq: Fr+ cont}}
F^+_r := r^{-1}\bigl(\{ e^\alpha \}\bigr), \quad
F^-_r := r^{-1}\bigl(\{ 1 \}\bigr),
\end{equation}
so that $dx$-almost everywhere we can write
\begin{equation}\label{eq: def r(x)}
    r(x) = e^\alpha \mathds{1}_{F^+_r}(x) + \mathds{1}_{F^-_r}(x)
    = 1 + \bigl(e^\alpha - 1\bigr) \mathds{1}_{F^+_r}(x).
\end{equation}
Finally, note that we equip $B$ {\modar with the metric topology of weak-$\star$ convergence and consider the associated Borel $\sigma$-algebra $\mathcal{B}(B)$, which makes the set $\mathcal{E}$ measurable (see details in \cite{amorino2025factorization}).}
\\
In this way, we obtain a construction that provides the analogue of a staircase pattern, as in Definition~\ref{def: staircase matrix}, but now in the infinite-dimensional setting. Our goal is to leverage this construction to describe an extremal (staircase) mechanism, with a particular focus on the symmetry of the associated measures. To do so, it is essential to evaluate elements of the set $B$ at a point $x \in \mathcal{X}$, which motivates the introduction of the following evaluation operator:
\[
e : 
\begin{cases}
(\mathcal{X} \times B, \mathcal{B}(\mathcal{X}) \otimes \mathcal{B}(B)) \to \mathbb{R}, \\
(x, r) \mapsto {\modar e_x(r),} 
\end{cases}
\]
{\modar which is }
measurable and satisfies, {\modar for all $r\in B $,} $e_x(r) = r(x)$ for almost every $x$.  

We are now ready to introduce extremal randomizations in the continuous setting. For any non-negative Borel measure $\mu$ on $\mathcal{E}$ satisfying
\begin{equation} \label{eq : cond norm mu continue}
    \text{for a.e.\ } x \in \mathcal{X}, \quad \int_{\mathcal{E}} e_x(r) \, \mu(dr) = 1,
\end{equation}
we define a randomization mechanism from $\mathcal{X}$ to $\mathcal{E}$ by means of the kernel
\begin{equation}{\label{eq: def mechanism measure}}
    q^{(\mu)}(x,dr) = e_x(r) \, \mu(dr).
\end{equation}
{\modar The condition \eqref{eq : cond norm mu continue} is here crucial to ensure that the measure $q^{(\mu)}(x,dr)$ defines a randomization
	on $\mathcal{E}$ for almost every $x\in\mathcal{X}$.}
	One can check that this randomization is extremal in the same sense as in~\eqref{eq : extremal constrainte discret}, since
\[
\frac{q^{(\mu)}(x,dr)}{q^{(\mu)}(x',dr)} 
= \frac{e_x(r)}{e_{x'}(r)} 
= \frac{r(x)}{r(x')} \in \{ e^{-\alpha}, 1, e^{\alpha} \}
\quad \text{for a.e.\ } dx \, dx'.
\]
Moreover, since $e_x(r) \in [1, e^\alpha]$, condition~\eqref{eq : cond norm mu continue} implies the bounds
\begin{equation}\label{eq: bound mu E}
    e^{-\alpha} \le \mu(\mathcal{E}) \le 1,
\end{equation}
which will be useful in the sequel.

\subsection{Parametric models and Fisher information}
In the following, it will be crucial to work with models satisfying appropriate regularity conditions. In this subsection, we begin by detailing the notion of regularity we have in mind.  

Assume that $\mathcal{X}$ is an open subset of $\mathbb{R}^d$ and that $(P_{\theta})_{\theta \in \Theta}$ is a family of probability measures on $\mathcal{X}$, with $\Theta \subset \mathbb{R}^{d_\Theta}$. We suppose that this family is dominated by the Lebesgue measure, and denote its density by
\[
p_{\theta}(x) := \frac{dP_\theta(x)}{dx}.
\]
Let $\alpha > 0$ and let $\mu$ be a Radon sub-probability measure on $\mathcal{E}$ satisfying \eqref{eq : cond norm mu continue}. We consider the associated privacy mechanism $q^{(\mu)}$ as introduced in \eqref{eq: def mechanism measure}, denote by $Z$ the corresponding public data, and by $(\widetilde{P}_{\theta})_{\theta \in \Theta}$ the induced model on $\mathcal{E}$, i.e., the law of $Z$.  

This model is regular in the sense of differentiability in quadratic mean. We recall here the definition of differentiability in quadratic mean (DQM) at an interior point $\theta_0 \in \overset{\circ}{\Theta}$.  

\begin{df}\label{def: DQM}
A statistical model $\mathcal{P} = (P_\theta)_{\theta \in \Theta}$, with $\Theta \subset \mathbb{R}^{d_\Theta}$, sample space {\modar $(\mathcal{Y}, \mathcal{G})$,} and dominating measure $\tilde{\mu}$, is said to be differentiable in quadratic mean (DQM) at $\theta_0 \in \overset{\circ}{\Theta}$ if the $\tilde{\mu}$-densities $p_\theta = \frac{dP_\theta}{d\mu}$ satisfy, for $h \in \mathbb{R}^{d_\Theta}$,
\[
\int_\mathcal{X} \left(
\sqrt{p_{\theta_0+h}(x)} - \sqrt{p_{\theta_0}(x)} - \tfrac{1}{2} h^{\top}s_{\theta_0}(x) \sqrt{p_{\theta_0}(x)}
\right)^2 \tilde{\mu}(dx) = o(\|h\|^2)
\quad \text{as } h \to 0,
\]
for some measurable vector-valued function $s_{\theta_0} : \mathcal{X} \to \mathbb{R}^{d_\Theta}$, which is called the score at $\theta_0$.
\end{df}

A model is said to be DQM if it is DQM at every point $\theta \in \overset{\circ}{\Theta}$. It will be crucial for our proofs to note that if the model $\mathcal{P}$ is DQM at $\theta_0$, then its score function is centered, that is
\begin{equation}\label{eq: score centered}
    \mathbb{E}_{\theta_0}\bigl[s_{\theta_0}\bigr] = 0,
\end{equation}
and the Fisher information exists, is finite, and takes the form
\[
\mathcal{I}_{\theta_0}(\mathcal{P}) = \mathbb{E}_{\theta_0}\bigl[s_{\theta_0} s_{\theta_0}^{\top}\bigr].
\]
A proof of these results can be found in Theorem~7.2 of \cite{Van07}.

If one considers transformations of DQM models via Markov channels satisfying local differential privacy, Lemma~3.1 of \cite{steinberger2024efficiency} shows that if the original model $(P_\theta)_\theta$ is DQM at $\theta_0$ {\modar with dominating measure $dx$}, then the privatized model $(q \circ P_\theta)_\theta$ is also DQM at $\theta_0$  {\modar with dominating measure $\mu$. The score function of the latter is} 
\begin{equation}\label{eq : score public}
    t_{\theta_0}(z) = \mathbb{E}\left[ s_{\theta_0}(X) \mid Z = z \right]
    = \frac{
        \int_\mathcal{X} s_{\theta_0}(x) \, q(x,z) \, p_{\theta_0}(x) \, \mu(dx)
    }{
        \int_\mathcal{X} q(x,z) \, p_{\theta_0}(x) \, \mu(dx)
    }.
\end{equation}
Moreover, according to Lemma~1 of \cite{amorino2025factorization}, the law of the public variable $Z$ that is the induced model $(\widetilde{P}_{\theta})_{\theta \in \Theta}$ on $\mathcal{E}$, admits a continuous density {\modar with respect to the reference measure $\mu$}
\begin{equation}\label{eq : tilde p ctn}
    \widetilde{p}_\theta(r) := \frac{d\widetilde{P}_{\theta}}{d\mu}(r)
    = \int_\mathcal{X} r(x) \, p_\theta(x) \, dx,
\end{equation}
which is differentiable in quadratic mean and has an explicit continuous score function
\begin{equation}\label{eq : score t ctn}
    t_{\theta_0}(r) = \frac{
        \int_\mathcal{X} s_{\theta_0}(x) \, r(x) \, p_{\theta_0}(x) \, dx
    }{
        \widetilde{p}_{\theta_0}(r)
    }
    = (e^\alpha - 1) \, \frac{
        \int_{F_r^+} s_{\theta_0}(x) \, p_{\theta_0}(x) \, dx
    }{
        \widetilde{p}_{\theta_0}(r)
    }.
\end{equation}
This allows us to write the Fisher information explicitly:
\begin{align}{\label{eq : Fisher ctn general_preli}}
    \mathcal{I}_{\theta_0}\bigl(q^{(\mu)} \circ \mathcal{P}\bigr)
    &= \mathbb{E}_{\theta_0}\left[
        t_{\theta_0}(Z) \, t_{\theta_0}(Z)^{\top}
    \right] \nonumber \\
    &= \int_\mathcal{E}
    t_{\theta_0}(r) \, t_{\theta_0}(r)^{\top} \,
    \widetilde{p}_{\theta_0}(r) \, \mu(dr) \nonumber \\
    &= \int_\mathcal{E}
    \frac{
        \left( \int_\mathcal{X} s_{\theta_0}(x) \, r(x) \, p_{\theta_0}(x) \, dx \right)
        \left( \int_\mathcal{X} s_{\theta_0}(x) \, r(x) \, p_{\theta_0}(x) \, dx \right)^{\top}
    }{
        \widetilde{p}_{\theta_0}(r)
    }
    \mu(dr) \\
    &=
    (e^\alpha - 1)^2
    \int_\mathcal{E}
    \frac{
        \left( \int_{F_r^+} s_{\theta_0}(x) \, p_{\theta_0}(x) \, dx \right)
        \left( \int_{F_r^+} s_{\theta_0}(x) \, p_{\theta_0}(x) \, dx \right)^{\top}
    }{
        \widetilde{p}_{\theta_0}(r)
    }
    \mu(dr). \nonumber
\end{align}
With all this background in place, we are finally ready to discuss the optimization problem for the Fisher information, which consists of searching for the privacy mechanism achieving maximal Fisher information. This search for the optimal privacy mechanism $q^{(\mu)}$ can be rephrased as looking for the optimal measure $\mu$ associated with such a mechanism.  

Let us consider the case where the parameter $\theta_0$ is one-dimensional, and define, for any $r \in B$,
\begin{equation}\label{eq : def i r ctn}
    i(r) := \frac{
        \left( \int_\mathcal{X} s_{\theta_0}(x) \, r(x) \, p_{\theta_0}(x) \, dx \right)^2
    }{
        \int_\mathcal{X} r(x) \, p_{\theta_0}(x) \, dx
    }
    = (e^\alpha - 1)^2
    \frac{
        \left( \int_{F_r^+} s_{\theta_0}(x) \, p_{\theta_0}(x) \, dx \right)^2
    }{
        1 + (e^\alpha - 1) \int_{F_r^+} p_{\theta_0}(x) \, dx
    },
\end{equation}
where we have used \eqref{eq: def r(x)}. This leads to the following optimization problem:
\begin{equation}\label{eq : optimisation pb ctn}
    I^* := \sup_{\mu}
    \int_\mathcal{E} i(r) \, \mu(dr),
    \quad
    \text{subject to } \int_\mathcal{E} e_x(r) \, \mu(dr) = 1
    \quad \text{for $dx$-almost every } x,
\end{equation}
where the supremum is taken over all Radon sub-probability measures $\mu$ on $\mathcal{E}$.  

From Theorem~2 of \cite{amorino2025factorization}, we know that this optimization problem admits a solution, i.e., there exists a Radon sub-probability measure $\overline{\mu}$ on $\mathcal{E}$ such that the mechanism defined by
\[
q^{(\overline{\mu})}(x,dr) := e_x(r) \, \overline{\mu}(dr)
\]
satisfies
\[
\mathcal{J}^{\text{max},\alpha}_{\theta_0}
= \mathcal{I}_{\theta_0}\left( q^{(\overline{\mu})} \circ \mathcal{P} \right).
\]
Although the existence of a privacy mechanism solving this optimization problem is guaranteed, little can be said about its explicit form.  

The main purpose of this paper is to investigate in greater detail this optimal Fisher information and to gain insights into the privacy mechanism achieving it. In particular, we will see that by introducing a notion of symmetry for the optimal measure, further structural insights can be obtained. 

\section{Main results}{\label{s: main}}

From the background developed in the previous section, it is clear that the structure of the set $\mathcal{E}$ plays a central role in the behavior of the Fisher information {\modar and on the definition of the privacy mechanism.}
	In both the discrete and continuous settings, the extremal values $1$ and $e^\alpha$ emerge as particularly significant. This naturally raises the question: can the analysis be simplified by considering measures on $\mathcal{E}$ that are, in some sense, indifferent to the distinction between these two values? Can this help to construct normalized Radon measures $\mu$ on $\mathcal{E}$?

To formalize this idea, let us consider Radon measures $\mu$ on $\mathcal{E}$ satisfying the normalization condition
\[
\int_\mathcal{E} e_x(r)\, \mu(dr) = 1 \quad \text{for a.e.\ } dx,
\]
and let us introduce the operator $T: \mathcal{E} \to \mathcal{E}$ defined by
\begin{equation}{\label{eq: def T sym}}
 T(r) := e^\alpha + 1 - r,   
\end{equation}
which essentially swaps the values $1$ and $e^\alpha$: for each $x$, we have $T(r)(x) = 1$ if and only if $r(x) = e^\alpha$. This operator captures a natural form of symmetry with respect to the extremal values in $\mathcal{E}$. Given a Radon measure $\mu$ on $\mathcal{E}$, we define the pushforward measure $T(\mu)$ by
\[
T(\mu) := \mu \circ T,
\]
so that for any integrable function $g$ on $\mathcal{E}$,
\[
\int_\mathcal{E} g(r)\, T(\mu)(dr) = \int_\mathcal{E} g(T(r))\, \mu(dr).
\]
{We begin by introducing the formal notion of symmetry that serves as a cornerstone for this work.
\begin{df}\label{def: sym mu}
Let $\mu$ be a Radon measure on $\mathcal{E}$ satisfying the normalization condition \eqref{eq : cond norm mu continue}. We say that $\mu$ is symmetric if it is invariant under the transformation $T$, i.e.,
\[
\mu = T(\mu).
\]
\end{df}
Although symmetry is defined via $T$-invariance, it is important to observe that $T(\mu)$ does not, in general, satisfy the normalization condition. To see this, we compute:
\begin{align*}
\int_\mathcal{E} e_x(r)\, T(\mu)(dr) 
&= \int_\mathcal{E} T(e_x(r))\, \mu(dr) 
= \int_\mathcal{E} \big(e^\alpha + 1 - e_x(r)\big)\, \mu(dr) \\
&= (e^\alpha + 1)\mu(\mathcal{E}) - \int_\mathcal{E} e_x(r)\, \mu(dr) .
\end{align*}
This discrepancy suggests that any arbitrary measure must be appropriately transformed to satisfy both symmetry and normalization simultaneously. As it turns out, any Radon measure $\mu$—irrespective of its initial properties—can be mapped to such a state via the following construction:
\begin{equation}\label{eq: mu(s)}
\mu^{(s)} := \frac{1}{(e^\alpha + 1)\mu(\mathcal{E})} \big( \mu + T(\mu) \big).
\end{equation}
The fact that $\mu^{(s)}$ is both normalized and symmetric (in the sense of Definition \ref{def: sym mu}) is established in the following proposition, the proof of which is provided in Section \ref{s: proof main}.}

\begin{prop}\label{prop: symmetrisée}
Let $\mu$ be a Radon measure on $\mathcal{E}$, and define $\mu^{(s)}$ as in \eqref{eq: mu(s)}. Then the following statements hold:
\begin{enumerate}
    \item The measure $\mu^{(s)}$ is normalized in the sense of \eqref{eq : cond norm mu continue} and symmetric in the sense of Definition~\ref{def: sym mu}.
    \item If the original measure $\mu$ is already symmetric (i.e., $\mu = T(\mu)$) {and normalized}, then $\mu(\mathcal{E}) = \frac{2}{e^\alpha + 1}$, and we recover $\mu^{(s)} = \mu$.
\end{enumerate}
\end{prop}
{\begin{rem}
It is worth emphasizing that the construction of $\mu^{(s)}$ is entirely self-correcting. The symmetry and normalization properties are enforced by the operator $T$ and the choice of the denominator in \eqref{eq: mu(s)}, regardless of whether the initial measure $\mu$ possesses any specific structure or even a unit mass.
\end{rem}}
Observe that, in the discrete setting, a key role is played by the two particular level sets:
\[
F_{\max}^+ := \left\{ x \in \mathcal{X} : s_{\theta_0}(x) > 0 \right\}
\quad \text{and} \quad
F_{\max}^{\prime +} := \left\{ x \in \mathcal{X} : s_{\theta_0}(x) < 0 \right\}.
\]
These sets are crucial because, in the numerator of the Fisher information, the following quantity appears (see Equation \eqref{eq : Fisher ctn general_preli}):
\[
\left( \int_\mathcal{X} s_\theta(x) \, \mathbf{1}_{F^+_r}(x) \, p_\theta(x) \, {\modar dx} \right)^2
\le
\left( \frac{1}{2} \int_\mathcal{X} |s_\theta(x)| \, p_\theta(x) \,  {\modar dx} \right)^2
= \left( \frac{1}{2} \mathbb{E}\bigl[|s_\theta(X)|\bigr] \right)^2,
\]
with equality {\modar if}   
\( F_r^+ = F_{\max}^+ \) or \( F_r^+ = F_{\max}^{\prime +} \).  

Among the collection of sets \( (F_r^+)_{r \in \mathcal{E}} \) introduced in~\eqref{eq: Fbeta+}, one can denote by \( r_{\max} \) and \( r_{\max}' \) the indices corresponding to \( F_{\max}^+ \) and \( F_{\max}^{\prime +} \), respectively. These provide natural candidates for the support of a solution to the optimization problem in the discrete case (see Remark~3 in \cite{amorino2025factorization} for further details). In particular, Proposition~7 of \cite{amorino2025factorization} uses this insight to establish an upper bound for the Fisher information, with a proof that is quite technical, showing that the maximum of the optimization problem is indeed attained by a sub-probability measure supported on the set \( \{ r_{\max}, r_{\max}' \} \). Proposition~8 of \cite{amorino2025factorization} is then devoted to proving a lower bound, and in particular explicitly constructs a privacy mechanism achieving this (optimal) Fisher information, based on the these two points.  

\begin{rem}{\label{rk: symmetry start}}
It is worth noting that the two-point measures discussed above, commonly used to establish the lower bound, illustrate the power of the symmetrized measure~$\mu^{(s)}$. Indeed, these are obtained by symmetrizing the naturally proposed choice
\[
\mu = \delta_r, \quad
\text{with } r = 1 + (e^\alpha - 1) \indi{F_{\max}^+}
\]
where \( F_{\max}^+ \) is as defined above. 
{\modar In Proposition 7 of \cite{amorino2025factorization}, we establish that this two points symmetric measure is optimal in a setting of a finite set $\mathcal{X}$ as soon as  $\alpha$ is small enough. Also, \cite{kalinin2024efficient} gives an example of two points symmetric mechanism optimal when $\mathcal{X}=\mathbb{R}$.}
\end{rem}

The previously established Proposition~\ref{prop: symmetrisée} enables the construction of a wide class of normalized and symmetric measures on $\mathcal{E}$, which in turn yield a variety of symmetric randomization mechanisms. The next result illustrates why such symmetry is particularly useful for our analysis: when the maximizer $\mu$ of the optimization problem \eqref{eq : optimisation pb ctn} (whose existence is guaranteed by Theorem 2 of \cite{amorino2025factorization}) is symmetric, the corresponding Fisher information attains an explicit and simplified expression.

To state this result, we first introduce a convenient notation. Recall the function $i(r)$ defined in \eqref{eq : def i r ctn}. Then, following the construction of the symmetrized measure $\mu^{(s)}$ in \eqref{eq: mu(s)}, we define the symmetrized version of $i$ as:
\begin{equation}\label{eq: i(s) r}
	i^{(s)}(r) := {\modar \frac{1}{2}} \left[ i(r) + i(T(r)) \right].
\end{equation}
This quantity plays a central role in simplifying the expression of the Fisher information under symmetry, as detailed in the theorem below.

\begin{thm}\label{thm: Fisher sym}
Assume that the maximizer $\bar{\mu}$ of the optimization problem \eqref{eq : optimisation pb ctn} is symmetric in the sense of Definition~\ref{def: sym mu}. Then, the function $r \mapsto i^{(s)}(r)$ is $\bar{\mu}$-almost everywhere constant, and the maximum Fisher information is given by
\begin{align*}
\mathcal{J}_{\theta_0}^{\max, \alpha} &= \frac{2}{e^\alpha + 1} (\max_{r \in B} i^{(s)}(r)), 
\end{align*}
with
\begin{equation}{\label{eq: i(s)r}}
i^{(s)}(r) = \frac{(e^\alpha - 1)^2}{2} \frac{(\int_{F_r^+}s_{\theta_0}(x)p_{\theta_0}(x) dx)^2}{(e^\alpha  -(e^\alpha - 1) \int_{F_r^+}p_{\theta_0}(x)dx)(1 + (e^\alpha - 1) \int_{F_r^+}p_{\theta_0}(x)dx)}.
\end{equation}
\end{thm}

We would like to compare our results on the maximal Fisher information with those previously established in the literature. However, very little is known about the form of the maximal Fisher information when the set $\mathcal{X}$ is infinite, apart from the existence of a solution. This motivates us to compare our explicit findings with results from the discrete setting, especially given that several similarities are immediately apparent.

When comparing the expression above with Theorem~1 in \cite{amorino2025factorization}, multiple parallels can be drawn. In particular, as discussed earlier, the numerator of the function we aim to maximize is always bounded above by
\[
\left( \frac{1}{2} \int_\mathcal{X} |s_{\theta_0}(x)| \, p_{\theta_0}(x) \, dx \right)^2
= \left( \frac{1}{2} \mathbb{E}\bigl[|s_{\theta_0}(X)|\bigr] \right)^2,
\]
with equality holding if and only if \( F_r^+ = F^+_{\max} \) or \( F_r^+ = F^{'+}_{\max} \). This property also appears in the optimization problem for the finite setting.  

An important difference with respect to the discrete case, however, lies in the structure of the index set \( \mathcal{E} \). In the discrete setting, \( \mathcal{E} = \{0, \dots, 2^d - 1\} \) is finite, which makes it possible to explicitly verify that the solution to the optimization problem is supported on only two points, by checking that the support of the optimal sub-probability vector has cardinality between~1 and~\( 2^d \). In contrast, in the continuous setting considered here, we have \( |\mathcal{E}| = \infty \), making such a verification no longer feasible. As a consequence, we cannot rigorously confirm that the candidate identified by maximizing the numerator indeed solves the optimization problem, without a more delicate analysis of the denominator.

\begin{rem}{\label{rem: 3}}
Even without the detailed analysis of the denominator just discussed, the result above highlights that if the optimal solution to the Fisher information maximization problem is symmetric, then an optimal mechanism can always be constructed as a binary mechanism. 
{\modar Indeed, it is established in the proof of Theorem \ref{thm: Fisher sym} that a maximizer of $r \in B\mapsto i^{(s)}(r)$ exists in $\mathcal{E}\subset B$. Taking $\bar{r}\in\mathcal{E}$ such maximizing value, we can consider $\tilde{\mu}=\frac{1}{1+e^{\alpha}} \left(\delta_{\bar{r}}+\delta_{T(\bar{r})}\right)$ which is a symmetric and normalized measure. Then, the Fisher information of $q^{\tilde{\mu}}$ is $\frac{2}{e^\alpha+1} i^{(s)}(\bar{r})$ and thus equal to $\mathcal{J}_{\theta_0}^{\max, \alpha}$.}
\end{rem}

Theorem \ref{thm: Fisher sym} provides a simplified expression for the Fisher information in the case where the solution measure to the optimization problem \eqref{eq : optimisation pb ctn} is symmetric. However, this symmetry assumption does not always hold. In what follows, we present a decomposition of the Fisher information into symmetric and asymmetric components. Before stating the result, let us introduce the notation for the asymmetric parts of the measure 
\[
\mu^{(as)} := \mu - \mu^{(s)}.
\]
With this notation, the following proposition holds (see Section \ref{s: proof main} for the proof).

\begin{prop}\label{p: decomp fisher sym}
Let \( \mu \) be a Radon measure satisfying the normalization condition \eqref{eq : cond norm mu continue}. Then, the Fisher information can be decomposed as
\begin{align*}
\mathcal{I}_{\theta_0}(q^{(\mu)} \circ \mathcal{P}) 
&= \int_{\mathcal{E}} i^{(s)}(r) \, \mu(dr) 
+ \left( \frac{e^\alpha + 1}{2} - \frac{1}{\mu(\mathcal{E})} \right) 2 \mu(\mathcal{E}) \int_{\mathcal{E}} i^{(s)}(r) \, \mu(dr) \\
&\quad -2 (e^\alpha - 1)\mu(\mathcal{E})\int_{\mathcal{E}} i^{(as)}(r) \, \mu(dr),
\end{align*}
where
\[
i^{(as)}(r) := i^{(s)}(r) \left( \int_{F_r^+} p_{\theta_0}(x) \, dx - \frac{1}{2} \right).
\]
The notation \((as)\) stands for "asymmetric", since this quantity satisfies the antisymmetry relation \( i^{(as)}(T(r)) = -i^{(as)}(r) \). Moreover, as \( \alpha \to 0 \), the Fisher information admits the expansion
\[
{\modar \mathcal{I}_{\theta_0}(q^{(\mu)} \circ \mathcal{P}) 
= 
\alpha^2\int_{\mathcal{E}} \left( \int_{F_r^+} s_{\theta_0}(x) p_{\theta_0}(x) \, dx \right)^2 \mu(dr) + O(\alpha^3).}
\]
\end{prop}

\begin{rem}{\label{rk: retrieve sym Fisher}}
As expected, if the solution measure is symmetric, we recover the result of Theorem \ref{thm: Fisher sym}, which corresponds precisely to the first term in the decomposition above. Indeed, in this case, we have \( \mu(\mathcal{E}) = \frac{2}{e^\alpha + 1} \) by Point 2 of Proposition \ref{prop: symmetrisée}, and the same proposition guarantees that \( \mu = \mu^{(s)} \). Moreover, since \( i^{(as)}(T(r)) = -i^{(as)}(r) \) and by the definition \eqref{eq: mu(s)} of the symmetrized measure \( \mu^{(s)} \), we obtain:
\[
\int_{\mathcal{E}} i^{(as)}(r) \, \mu^{(s)}(dr) = \int_{\mathcal{E}} \frac{1}{2} \left[ i^{(as)}(r) + i^{(as)}(T(r)) \right] \mu(dr) = 0.
\]
This shows that both the second and third terms in the decomposition vanish when the measure is symmetric.
\end{rem}

\begin{rem}\label{rk: sym fisher alpha small}
In scenarios where the data are highly sensitive and thus a high level of privacy is required — corresponding to the regime \( \alpha \to 0 \) — the asymmetric contribution to the Fisher information becomes negligible compared to the symmetric one. Indeed, the leading term in the decomposition, corresponding to the symmetric contribution, is of order \( \alpha^2 \), as seen in the asymptotic expansion. The remaining two terms, which capture the asymmetric contribution, also involve integrals of order \( \alpha^2 \), but are multiplied by prefactors of order \( \alpha \), namely \( \left( \frac{e^\alpha + 1}{2} - \frac{1}{\mu(\mathcal{E})} \right) \) and \( e^\alpha - 1 \). As a result, their overall contribution becomes of order \( \alpha^3 \), and is therefore negligible in the limit \( \alpha \to 0 \). 
\end{rem}
This final observation motivates our focus on symmetric measures. In the high-privacy regime, which is the most relevant from an applied perspective, the Fisher information is dominated by its symmetric component, while the contribution of the asymmetric part vanishes asymptotically. 

This behavior no longer holds in the low-privacy regime, corresponding to large values of \( \alpha \). As detailed in the next section, which is devoted to the analysis of asymmetric privacy mechanisms, asymmetry plays a non-negligible role.

{\subsection{Asymmetric privacy mechanisms}{\label{ss:asymm}}

For large values of $\alpha$, neglecting the asymmetric part of the measure to focus solely on the symmetric component no longer appears to be an appropriate approach. To emphasize this, we now propose to analyze a purely asymmetric case, corresponding to a modified version of the staircase mechanism which presents peaks centered around \( x \).
We assume that $\mathcal{X}=\mathbb{R}$ and construct the mechanism as follows.
First, we {\modar consider} two increasing functions $g : \mathbb{R} \to \{-\infty\} \cup \mathbb{R}$ and $d:  \mathbb{R} \to \mathbb{R} \cup \{\infty\}$ such that $g(x) < d(x)$. We define the mapping 
\begin{equation} 
	\label{eq : def I general}
	I: \left\{
	\begin{array}{llll} 
		&\mathcal{X} &\to& \mathcal{E} \\
&x_0 &\mapsto &I(x_0) = r_{x_0},
	\end{array}
	\right. 
\end{equation} 
where the function $r_{x_0}$ is the element of $\mathcal{E}$ defined by
\begin{equation}\label{eq : def r x0}
	r_{x_0} : \left\{
	\begin{array}{llll}
	&\mathbb{R} &\to & \{1,e^\alpha\}
	\\
	&x&\mapsto &r_{x_0}(x) = 1 + (e^\alpha - 1) \indi{[g(x_0), d(x_0)]}(x).
	\end{array}
		\right. 
\end{equation}
Note that the functions $r_{x_0}$ take the constant value \( 1 \) across the entire space $\mathcal{X}$, except in a the interval of $\mathbb{R}$ defined by $ [{g}(x_0), {d}(x_0)]$, where their value rises to $ e^\alpha $. Remark that when  an edge of the interval is not finite, $g(x_0)=-\infty$ or $d(x_0)=\infty$, the interval should be considered opened at this edge.   
Let $\nu$ be a probability measure on $\mathcal{X}$ that is absolutely continuous with respect to the Lebesgue measure, i.e. it admits a density: $\nu(dx) = \nu(x)\, dx$. Define $\mu^{\modar0}$ on $\mathcal{E}$ as the pushforward of $\nu$ through the map $I$, that is, $\mu^{\modar0} = I_\# \nu$. Hence, for any measurable function \( f: \mathcal{E} \to \mathbb{R} \),
\begin{equation}\label{eq: mu kernel}
	\int_{\mathcal{E}} f(r)\, \mu^{\modar0}(dr) = \int_\mathcal{X} f\big( I(x_0) \big)\, \nu(x_0)\, dx_0.
\end{equation}
Observe that, in general, $\mu^{\modar0}$ does not satisfy the normalization condition \eqref{eq : cond norm mu continue}. However, in some cases, it is possible to rescale it and thereby introduce a new probability measure satisfying it. 

Let us explain below which conditions on $g$ and $d$ are sufficient in order to rescale the measure $\mu^{\modar0}$ into a normalized measure according to definition
\eqref{eq : cond norm mu continue}.

First, we assume that $g$ is left continuous while $d$ is right continuous on $\mathbb{R}$ and that $g$ takes the value $-\infty$ on some interval $(-\infty,x^*_g]$
and that  
$d$ takes the value $\infty$ on some interval $[x^*_d,\infty)$.
We define the inverses of the functions $g$ and $d$ as follows:
\begin{equation*}
	g^{-1}(x)=\sup \{ y \in \mathbb{R} \mid g(y) \le x \}, 
	\quad
	d^{-1}(x)=\inf \{ y \in \mathbb{R} \mid d(y) \ge x \}. 
\end{equation*}
With these definitions, and using left (resp.\ right) continuity of $g$ (resp.\ $d$), we have for all $(x,x_0) \in \mathbb{R}^2$
\begin{equation}\label{eq: inverse function l r}
	\mathbbm{1}_{\{g(x_0) \le x \le d(x_0)\}} 
	= \mathbbm{1}_{\{d^{-1}(x) \le x_0 \le g^{-1}(x)\}}.
\end{equation}
\begin{lemma}\label{l : first step staircase}
	Let $c \in (0,1)$ and assume that 
    \begin{equation}{\label{eq: cond 28.5}}
    \int_{d^{-1}(x)}^{g^{-1}(x)} \nu(x_0)dx_0=c
    \end{equation}
    for all $x \in \mathbb{R}$.
	Define 
	\begin{equation*}{\modar \mu} := \frac{\mu^{\modar0}}{1 + (e^\alpha - 1)c}.
	\end{equation*}
	 Then, the measure ${\modar \mu}$ satisfies the condition \eqref{eq : cond norm mu continue}. 
	\end{lemma}
	\begin{proof} For $x \in \mathbb{R}$, we write
		\begin{align*}
			\int_{\mathcal{E}} e_{x}(r)\, \mu^{\modar0}(dr) 
			&= \int_\mathcal{X} e_{x}\big( I(x_0) \big)\, \nu(x_0)\, dx_0,
\quad \text{having used \eqref{eq: mu kernel},}			
			 \\
			 &= \int_\mathcal{X} r_{x_0}(x) \, \nu(x_0)\, dx_0,
			 \quad \text{from the definition of $I$,}			
			 \\
			&= \int_\mathcal{X} \big[ 1 + (e^\alpha - 1)\, 
			\mathbbm{1}_{[g(x_0), d(x_0)]}(x) \big] \nu(x_0)\, dx_0
			,\quad \text{where we used \eqref{eq : def r x0},}		
			 \\
			&= 1 + (e^\alpha - 1) \int_\mathcal{X} 
			\mathbbm{1}_{ [g(x_0), d(x_0)]}(x)\, \nu(x_0)\, dx_0.
		\end{align*}
	From \eqref{eq: inverse function l r}, 
	and  $\int_{d^{-1}(x)}^{g^{-1}(x)} \nu(x_0)dx_0=c$,  we deduce
	\begin{equation*}
		\int_{\mathcal{E}} e_{x}(r)\, \mu^{\modar0}(dr)  
		=  1 + (e^\alpha - 1) \int_\mathcal{X} 
		\mathbbm{1}_{[ d^{-1}(x), g^{-1}(x)]}(x_0) \nu(x_0) \, dx_0
		= 1   + (e^\alpha - 1)c.
	\end{equation*}	
	This show that {\modar $\mu=\mu^0/(1   + (e^\alpha - 1)c)$} satisfies the renormalization condition \eqref{eq : cond norm mu continue}.
	\end{proof}
We now propose a construction of the two functions $g$ and $d$ such that Lemma \ref{l : first step staircase} applies.

Assume that $\nu(x)>0$ for all $x \in \mathbb{R}$ and set $\Xi(x)=\int_{-\infty}^x \nu(y)dy$. We fix $c\in (0,1)$. We define the functions $g_c$ and $d_c$ as follows
\begin{equation}\label{eq : def lc}
	g_c(x)=
	\begin{cases}
		-\infty & \text{if $x \le x^*_{g,c}:=\Xi^{-1}(c)$}
		\\
		\Xi^{-1}(-\frac{c}{2}+\Xi(x)) & \text{if $x >x^*_{g,c}$} 
	\end{cases}
\end{equation}
\begin{equation}\label{eq : def rc}
	d_c(x)=
	\begin{cases}
		\Xi^{-1}(\frac{c}{2}+\Xi(x)) & \text{if $x < x^{*}_{d,c}:=\Xi^{-1}(1-c)$}
\\			\infty & \text{if $x \ge x^{*}_{d,c}$} 
	\end{cases}
\end{equation}
These functions are increasing and with $g_c (x) < d_c(x)$ for all $x$, as $\Xi^{-1}$ is strictly increasing on $(0,1)$ and $c>0$. By definition $g_c$ is left continuous and $d_c$ is right continuous. The next lemma states that these functions satisfy the constraint in \eqref{eq: cond 28.5}. Its proof can be found in Section \ref{s: proof main}

\begin{lemma}\label{l : second step staircase}
We have for all $x\in \mathbb{R}$,
\begin{equation*}
 \int_{d_c^{-1}(x)}^{g_c^{-1}(x)}\nu(x_0) \, dx_0 = c.
\end{equation*}
\end{lemma}

%
%
%

%
%
%
%
%

Now, we consider the measure $\mu^{\modar0}$ defined through $\mu^{\modar0} = I_\# {\modar \nu}$,
where we recall the definition of the mapping $I$ in \eqref{eq : def I general}--\eqref{eq : def r x0}, with the choice of function $g$ and $d$ given 
by \eqref{eq : def lc}--\eqref{eq : def rc}.
Using Lemmas \ref{l : first step staircase}--\ref{l : second step staircase} 
the measure \( \mu^{\modar0} \) 
can be rescaled to define the measure on $\mathcal{E}$
\begin{equation} \label{eq : mu_0 mu}
{\modar \mu} := \frac{\mu^{\modar0}}{1 + (e^\alpha - 1)c},
\end{equation}
which satisfies the normalization condition \( \int_\mathcal{E} e_{x_0}(r)\, {\modar \mu}(dr) = 1 \).  
The associated privacy mechanism
\[
q^{({\modar \mu})}(x,dr) := e_{\modar x}(r)\, {\modar \mu}(dr)
\]
can thus be rewritten as
\begin{equation}\label{eq: q kernel}
q^{(\mu)}(x,dr)
	= e_{\modar x}(r)\, \frac{\mu^{\modar0}(dr)}{1 + (e^\alpha - 1)c }.
\end{equation}

The following proposition shows that this 
construction yields an asymptotically efficient mechanism in the regime where \( \alpha \to \infty \) and \( c \to 0 \), up to ask some regularity on the model and on $\nu_0$.

\begin{ass} \label{Ass : smoothness for staircase} 
	We assume that statistical model \( (P_{\theta})_{\theta \in \Theta} \) is differentiable in quadratic mean (DQM) {\modar at \( \theta_0 \).} 
	 Moreover we assume that $x \mapsto {p}_{\theta_0}(x)$ and $x \mapsto s_{\theta_0}(x)$ are continuous at almost every point $x\in\mathbb{R}$. We assume that {\modar $\nu$} is a continuous function. 
\end{ass}

\begin{prop}\label{prop: kernel}
	Suppose that Assumption \ref{Ass : smoothness for staircase} holds true.
	Then, for the privacy mechanism \( q^{({\modar \mu})} \) defined in \eqref{eq: q kernel}, we have
	\[
	\mathcal{I}_{\theta_0}\big(q^{({\modar \mu})} \circ \mathcal{P}\big)
\longrightarrow
	\int_\mathcal{X} \frac{\dot{p}_{\theta_0}(x)^2}{p_{\theta_0}(x)}\, dx
	\qquad \text{as } c \to 0, \, c e^\alpha \to \infty.
	\]
\end{prop}
\begin{rem}
	\begin{enumerate}
		\item The privacy mechanism \eqref{eq: q kernel} is  
		asymmetric, and typically ${\modar\mu}$ and $T({\modar \mu})$ have disjoint support.
		\item This randomization procedure encodes the information on the 
		private data $x_0$ through the construction of the function $r_{x_0}$,
		 {\modar which exhibits} 
		 a peak centered around \( x_0 \) which shifts with \( x_0 \).
		 \item  In a simpler context where $\mathcal{X}=\mathbb{R}/\mathbb{Z} \sim [0,1)$ is the torus and {\modar $\nu$} is uniform on $[0,1)$, we can choose $g(x)=x-c/2$ and $d(x)=x+c/2$ and define $\mu^{\modar 0}=I_\sharp({\modar \nu})$,  considering all functions on $\mathcal{X}$ as $1$-periodic. 		 
		 Then, it is possible, along the lines of the proof of Lemma \ref{l : first step staircase}, to check that the measure $\mu^0$ defined in \eqref{eq : mu_0 mu} satisfies 
		 \eqref{eq : cond norm mu continue}. Actually, the construction above extends this example to the real line $\mathcal{X}=\mathbb{R}$ and allows the choice of a general probability measure $\nu_0$.
	\end{enumerate}
\end{rem}
}
\subsection{Implementation of the method: from theory to practice}
\label{ss: Implementation}

We now turn to the practical application of the privacy mechanisms introduced in Section \ref{ss:asymm}. While the theoretical construction of the measure $q^{(\mu)}(x,dr)$ in \eqref{eq: def mechanism measure} may appear abstract, we show here that it yields a concrete, implementable algorithm.

Our strategy relies on transforming the abstract output space $\mathcal{E}$ into a manageable, finite-dimensional representation. We define the image of our mapping as $\widehat{\mathcal{E}}=\{I(x_0) \mid x_0\in \mathbb{R}\} \subset \mathcal{E}$.

A central feature of our construction is the parameter $c$, which acts as a tuning knob for the mechanism's behavior. 

First, we ensure that our mapping preserves the data structure.
\begin{lemma}
	For any privacy level $\alpha>0$, if we choose the tuning parameter $c \le 1/2$, the map $I : \mathcal{X} \to \widehat{\mathcal{E}}$ is one-to-one.
\end{lemma}

\begin{proof}
	For any $x_0 \in \mathbb{R}$, the function $I(x_0)$ is defined as $x\mapsto 1 + (e^\alpha-1) \indi{[g_c(x_0),d_c(x_0)]}(x)$.
    Injectivity is guaranteed if the interval boundaries $g_c(x_0)$ and $d_c(x_0)$ uniquely identify $x_0$.
	From definitions \eqref{eq : def lc}--\eqref{eq : def rc}, and given that the base measure $\nu>0$, the functions $x_0 \mapsto g_c(x_0)$ and $x_0 \mapsto d_c(x_0)$ are strictly increasing where finite.
    Crucially, for $c\le 1/2$, at least one boundary is always finite, ensuring that the pair $(g_c(x_0),d_c(x_0))$ uniquely recovers $x_0$.
\end{proof}

From this point forward, we fix $c \le 1/2$. By the Lusin-Souslin theorem, $\widehat{\mathcal{E}}$ is a Borel set and $I^{-1}$ is measurable, allowing us to pull back the randomized value into the data space $\mathcal{X}$. We define the probability measure on $\mathcal{X}$ as $\widehat{q}^{(\mu)}(x,dx_0) = I^{-1}_\sharp(q^{(\mu)}(x,dr))$. We can derive the explicit density of this mechanism. For any test function $f$, the integral transformation yields:
\begin{align*}
\int_\mathbb{R}f(x_0)	\widehat{q}^{(\mu)}(x,dx_0) & =  
\int_\mathbb{R}f(x_0)	I^{-1}_\sharp(q^{(\mu)}(x,dr)) (dx_0)=
\int_\mathbb{\widehat{\mathcal{E}}}f(I^{-1}(r))	\widehat{q}^{(\mu)}(x,dr)
\\
&=
\int_\mathbb{\widehat{\mathcal{E}}}f(I^{-1}(r))	e_{x}(r) \frac{\mu^0(dr)}{1+c(e^\alpha-1)}, \quad \text{from \eqref{eq: q kernel},}
\\
&=
\int_\mathbb{\mathcal{X}}f(I^{-1}(I(x_0)))	e_x(I(x_0)) \frac{\nu(x_0)dx_0}{1+c(e^\alpha-1)},
\quad \text{as $\mu^0 = I_\sharp(\nu)$,}
\\
&=\int_\mathbb{\mathcal{X}}f(x_0) \left( 1+ (e^{\alpha}-1) \indi{[g_c(x_0),d_c(x_0)]}(x) \right) \frac{\nu(x_0)dx_0}{1+c(e^\alpha-1)},
\end{align*}
where we used the expression of the function $I(x_0) \in \mathcal{E}$ and the fact that $e_x$ is the evaluation operator at $x$.
Recalling \eqref{eq: inverse function l r}, we find the following explicit expression for density of the randomization mechanism, when $c \leq 1/2$,
\begin{equation}\label{eq : density rando X to X}
	\frac{\widehat{q}^{(\mu)}(x,dx_0)}{dx_0} = 
		 \left( 1+ (e^{\alpha}-1) \indi{[d_c^{-1}(x),g_c^{-1}(x)]}(x_0) \right) \frac{\nu(x_0)}{1+c(e^\alpha-1)}.
\end{equation}
The structure of \eqref{eq : density rando X to X} suggests a straightforward implementation using rejection sampling with proposal density $\nu$. Relative to the proposal $\nu$, the mechanism outputs public data $x_0$ more frequently when it falls within the interval $[d_c^{-1}(x),g_c^{-1}(x)]$, which contains the private data $x$. As $c \to 0$, this high-probability interval shrinks to the singleton $\{x\}$, minimizing noise; conversely, as $c$ increases, the interval widens, increasing privacy.

If the private data follows a density $p_\theta(x)$ with cumulative distribution function $F_\theta$, the resulting density of the public output $\widehat{p}_\theta(x_0)$ is derived by integrating over the private data {\modar distribution}:
\begin{align*}
	\widehat{p}_\theta(x_0)&=\int_\mathbb{R}	\frac{\widehat{q}^{(\mu)}(x,dx_0)}{dx_0} p_\theta(x)dx
	\\
	&= \frac{\nu(x_0)}{1+c(e^\alpha-1)} \left[1+(e^\alpha-1)(F_\theta(d_c(x_0)) - F_\theta(g_c(x_0))) \right],
\end{align*}
having used \eqref{eq : density rando X to X} and \eqref{eq: inverse function l r}.

\begin{rem}
    We restrict our implementation to $c \le 1/2$ for a critical reason: information collapse, as discussed in the following items.
    \begin{itemize}
        \item Loss of injectivity: when $c > 1/2$, the mapping $I$ is no longer one-to-one. For a wide range of central data values (specifically $x_0 \in (\Xi(1-c),\Xi(c))$), the mechanism outputs a single constant $I(x_0) = \mathfrak{r} := 1+(e^\alpha-1) \indi{\mathbb{R}}$. Effectively, the mechanism becomes "blind," failing to distinguish between any data points in this region.
        \item Statistical cost: this structural collapse destroys utility. Because the "blind spot" carries zero Fisher information ($i(\mathfrak{r})=0$), the overall information of the mechanism drops significantly.
        \item Numerical confirmation: as shown in Figures \ref{fig:std_theo_N_O_1}--\ref{fig:std_theo_Cauchy}, the standard deviation increases immediately when $c$ exceeds $1/2$. Thus, $c=1/2$ represents the structural boundary for the optimality of this method.
    \end{itemize}
\end{rem}

To assess the impact of the tuning parameter $c$ and the density $\nu$, we simulate a Gaussian translation model $X\sim\mathcal{N}(\theta,1)$ with $\theta_0=0$ using $n=10^3$ observations and $2 \times 10^3$ Monte Carlo trials. We test two proposal measures $\nu$: a standard Gaussian and a Cauchy distribution. As illustrated in Figures \ref{Figure : Gaussian}--\ref{Figure : Cauchy}, the theoretical standard deviation increases significantly for $c \ge 1/2$; consequently, we restrict our analysis to the efficient regime $c \in (0, 1/2]$ where empirical results perfectly match theoretical predictions. Our simulations reveal a clear transition in optimality based on the privacy level: in the high-privacy regime ($\alpha \le 1$), the choice $c=1/2$ minimizes variance, as lower values of $c$ yield negligible Fisher information. Conversely, in the low-privacy regime ($\alpha$ large), efficiency requires asymmetry; for example, at $\alpha=4$ with a Gaussian $\nu$, the minimal standard deviation of $3.67 \times 10^{-2}$ is achieved at $c \simeq 0.2$, a result that closely approaches the non-private MLE benchmark of $3.16 \times 10^{-2}$.

\begin{figure}
	\begin{subfigure}{.5\textwidth}
		\centering
		\includegraphics[width=.8\linewidth]{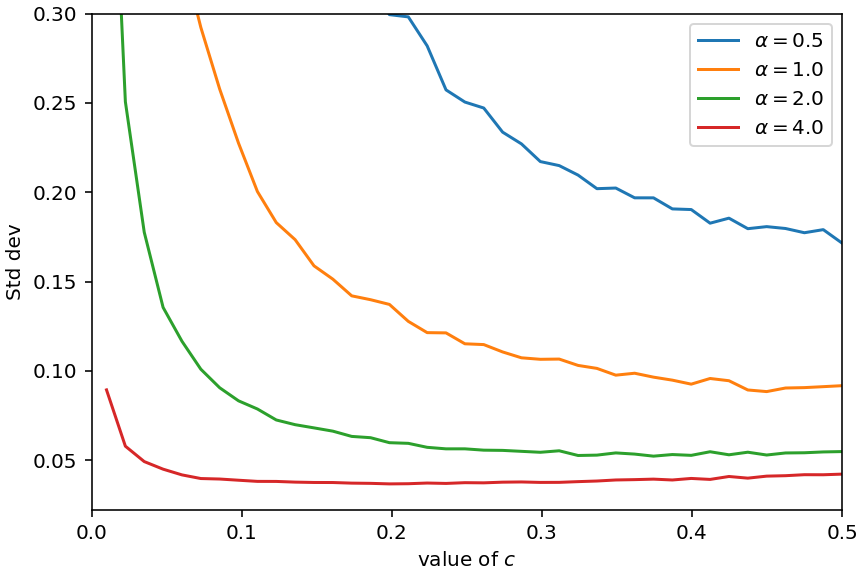}
		\caption{Standard deviation}
		\label{fig:std_N_O_1}
	\end{subfigure}%
	\begin{subfigure}{.5\textwidth}
		\centering
		\includegraphics[width=.8\linewidth]{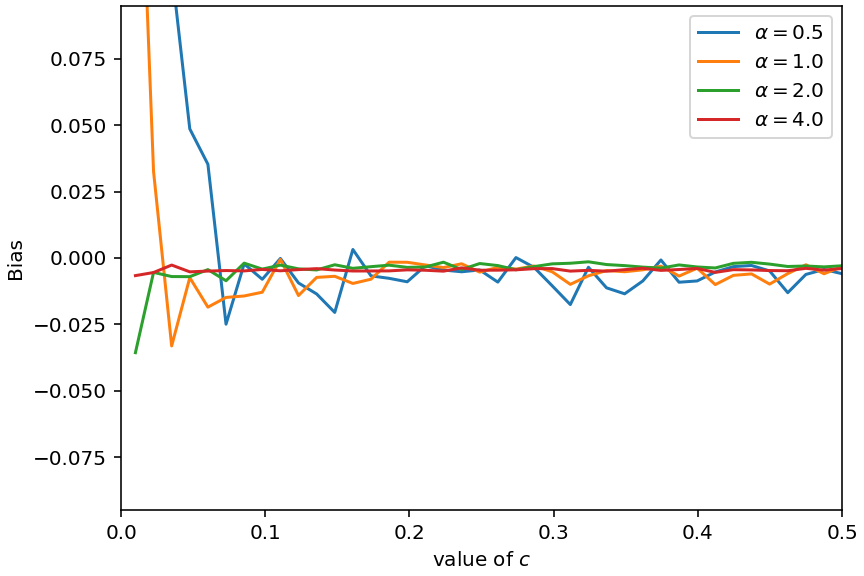}
		\caption{Bias}
		\label{fig:bias_N_O_1}
	\end{subfigure}
	\\
	\begin{subfigure}{.5\textwidth}
		\centering
		\includegraphics[width=.8\linewidth]{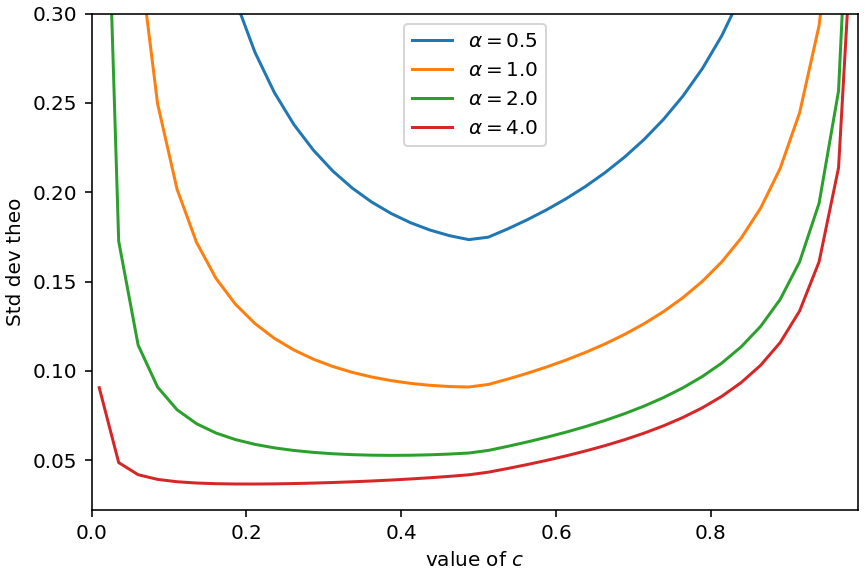}
		\caption{Theoretical standard deviation}
		\label{fig:std_theo_N_O_1}
	\end{subfigure}
	\caption{$\nu(x)=\frac{1}{\sqrt{2\pi}}e^{-x^2/2}$}
	\label{Figure : Gaussian}
\end{figure}

\begin{figure}
	\begin{subfigure}{.5\textwidth}
		\centering
		\includegraphics[width=.8\linewidth]{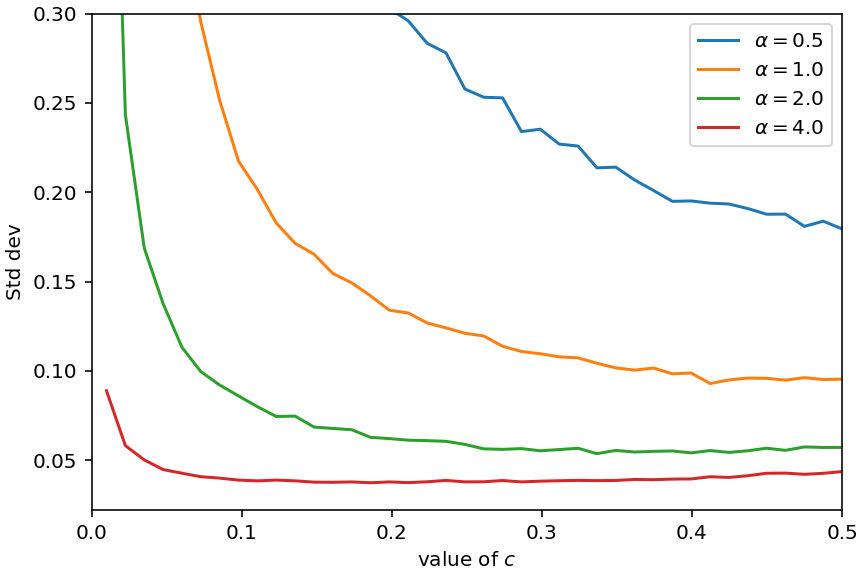}
		\caption{Standard deviation}
		\label{fig:std_Cauchy}
	\end{subfigure}%
	\begin{subfigure}{.5\textwidth}
		\centering
		\includegraphics[width=.8\linewidth]{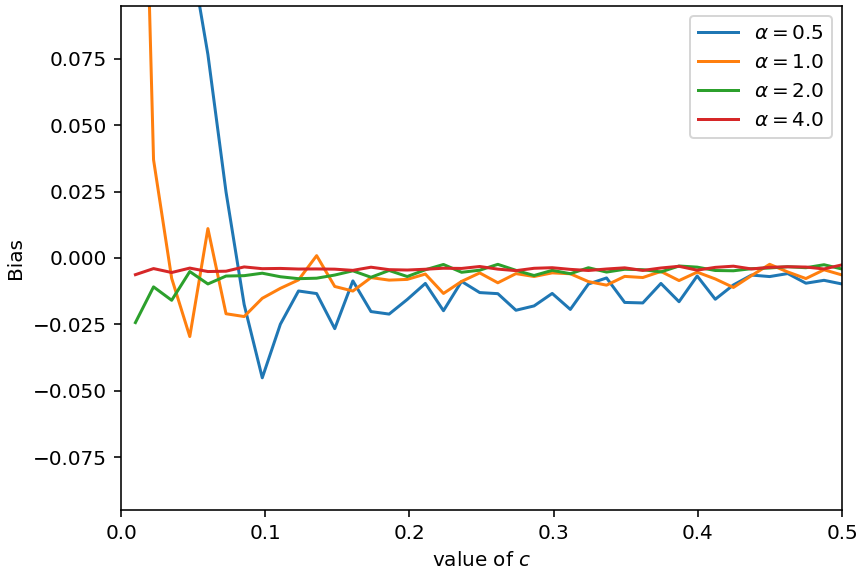}
		\caption{Bias}
		\label{fig:bias_Cauchy}
	\end{subfigure}
	\\
	\begin{subfigure}{.5\textwidth}
		\centering
		\includegraphics[width=.8\linewidth]{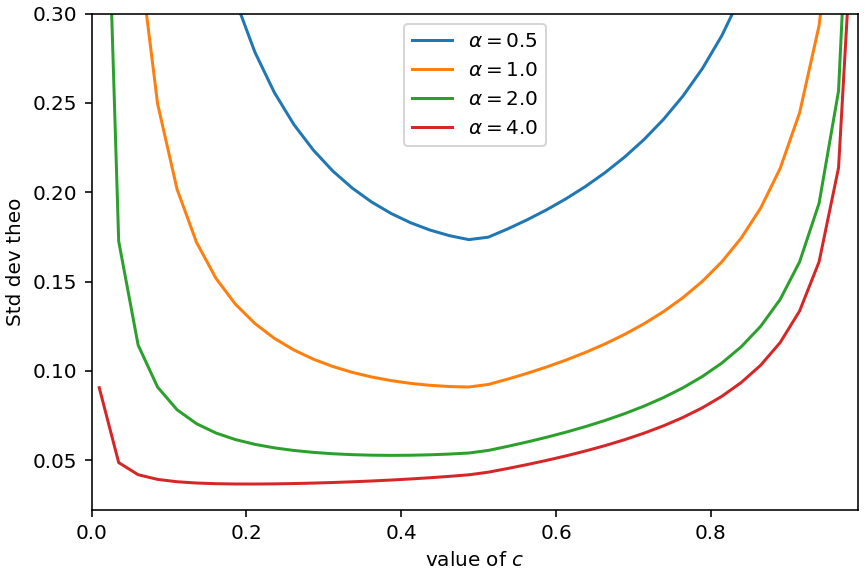}
		\caption{Theoretical standard deviation}
			\label{fig:std_theo_Cauchy}
	\end{subfigure}
	\caption{$\nu(x)=\frac{1}{\pi(1+x^2)}$}
	\label{Figure : Cauchy}
\end{figure}

\subsection{Approximation of binomial mechanism}{\label{ss:approx}}
To bridge the gap between the mechanism of Section \ref{ss:asymm} (optimal for low privacy) and the binomial mechanism (known to be asymptotically optimal for high privacy, $\alpha \to 0$), we introduce a modified construction that seamlessly interpolates between these two regimes. By adjusting the parameter $c$, this generalized mechanism can transition from high-precision local estimation to a coarse binary classification. We assume a unimodal location model where the score function's sign partitions the space into positive and negative half-lines, specifically $F^+_{\text{max}}=(0,\infty)$ and $F^{\prime +}_{\text{max}}=(-\infty,0)$. We assign positive densities $\nu^+$ and $\nu^-$ to these domains, with respective cumulative distribution functions $\Xi^+(x)=\int_0^x \nu^+(y)dy$ and $\Xi^-(x)=\int_{-\infty}^x \nu^-(y)dy$. To construct the mechanism, we define the boundary functions for $x>0$ as
\begin{align*}
	g^+_c(x)&=
	\begin{cases}
		0 & \text{ if $x \in (0, (\Xi^+)^{-1}(c) ]$}
			\\
		(\Xi^+)^{-1}(-\frac{c}{2}+\Xi^+(x)) & \text{ if $x \in ( (\Xi^+)^{-1}(c), \infty)$}
	\end{cases}
	\\
	d^+_c(x)&=
	\begin{cases}
		(\Xi^+)^{-1}(\frac{c}{2}+\Xi^+(x)) & \text{ if $x \in (0 , (\Xi^+)^{-1}(1-c))$}
		\\
		+\infty & \text{ if $x \in [ (\Xi^+)^{-1}(1-c), \infty)$}
	\end{cases}
\end{align*}
and symmetrically for $x<0$,
\begin{align*}
	g^-_c(x)&=
	\begin{cases}
		-\infty  & \text{ if $x \in (-\infty, (\Xi^-)^{-1}(c) ]$}
		\\
		(\Xi^-)^{-1}(-\frac{c}{2}+\Xi^-(x)) & \text{ if $x \in ( (\Xi^-)^{-1}(c),0)$}
	\end{cases}
	\\
	d^-_c(x)&=
	\begin{cases}
		(\Xi^-)^{-1}(\frac{c}{2}+\Xi^-(x)) & \text{ if $x \in (-\infty , (\Xi^-)^{-1}(1-c))$}
		\\
		0 & \text{ if $x \in [ (\Xi^-)^{-1}(1-c), 0)$.}
	\end{cases}
\end{align*}
By defining the interval boundaries $g_c(x_0)$ and $d_c(x_0)$ as the superposition of these positive and negative components, we construct the measure $\mu = I_\sharp(\frac{\nu_+ + \nu_-}{2+c(e^\alpha-1)})$, which we verify satisfies the normalization condition \eqref{eq : cond norm mu continue}. The crucial feature of this construction is its limit behavior: as $c \to 1$, the noise intervals $(g^\pm_c, d^\pm_c)$ expand to fill the entire half-lines $(0, \infty)$ and $(-\infty, 0)$. Consequently, the measure $\mu$ converges to the symmetric binomial measure $\frac{1}{1+e^\alpha}\left(\delta_{r_\text{max}} + \delta_{r'_\text{max}}\right)$, effectively recovering the binary output structure required for optimality at small $\alpha$. Conversely, for small $c$, the mechanism retains the local peak structure of Section \ref{ss:asymm}, preserving information for large $\alpha$. In our numerical experiments using folded normal distributions for $\nu^\pm$ (Figure \ref{Figure : Bino_Approx}), this flexibility yields tangible gains: for the high-privacy setting $\alpha=0.5$, choosing $c \simeq 1$ reduces the theoretical standard deviation to $0.162$, an improvement of approximately 5\% over the purely continuous Gaussian mechanism from Figure \ref{fig:std_N_O_1}.

\begin{figure}
	\begin{subfigure}{.5\textwidth}
		\centering
		\includegraphics[width=.8\linewidth]{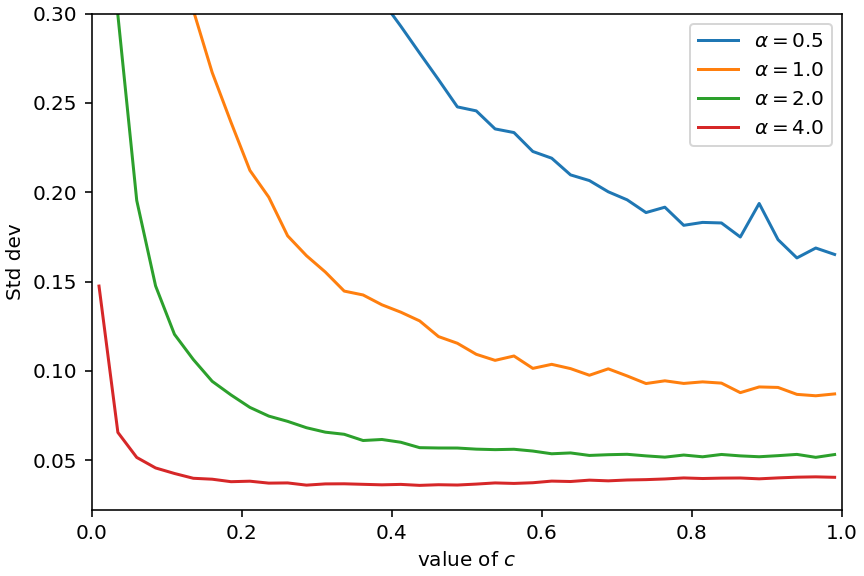}
		\caption{Standard deviation}
		\label{fig:std_Bino_Approx}
	\end{subfigure}%
	\begin{subfigure}{.5\textwidth}
		\centering
		\includegraphics[width=.8\linewidth]{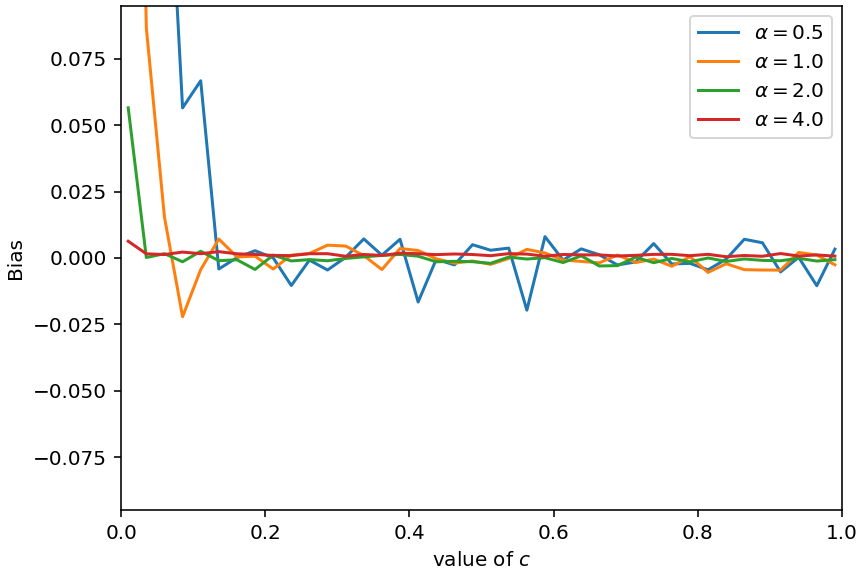}
		\caption{Bias}
		\label{fig:bias_Bino_Approx}
	\end{subfigure}
	\\
	\begin{subfigure}{.5\textwidth}
		\centering
		\includegraphics[width=.8\linewidth]{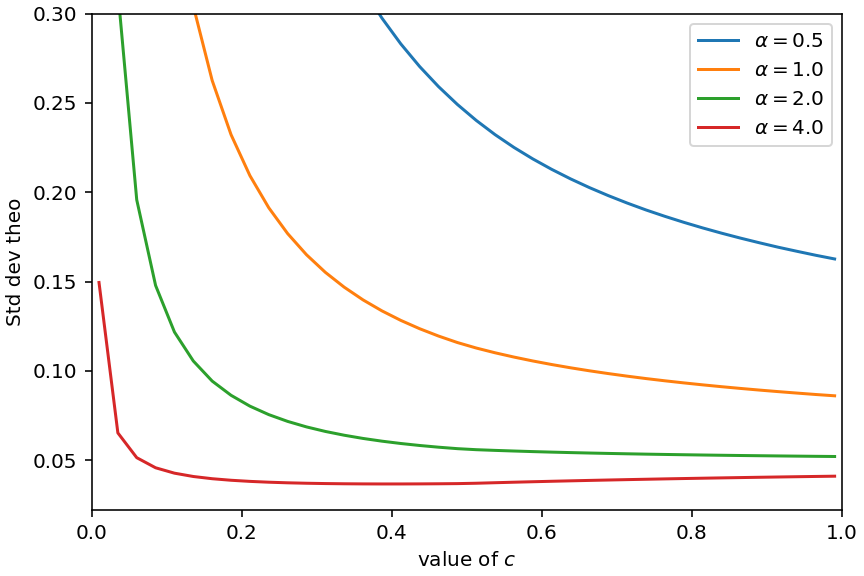}
				\caption{Theoretical standard deviation}
		\label{fig:std_theo_Bino_Approx}
	\end{subfigure}
	\caption{Approximation of binomial mechanism}
	\label{Figure : Bino_Approx}
\end{figure}

\section{Proofs}{\label{s: proof main}}

This section is dedicated to the proofs of our main results.

\subsection{Proof of Proposition \ref{prop: symmetrisée}}
\begin{proof}
\begin{enumerate}
    \item {First, we verify the normalization of the measure $\mu^{(s)}$. By definition, the integral of the function $e_x$ with respect to $\mu^{(s)}$ is given by:
\begin{equation*}
\int_\mathcal{E} e_x(r) \, \mu^{(s)}(dr) = \frac{1}{(e^\alpha + 1)\mu(\mathcal{E})} \int_\mathcal{E} e_x(r) \, \big(\mu + T(\mu)\big)(dr).
\end{equation*}
To evaluate the contribution of the second term, we utilize the properties of the operator $T$:
\begin{align}\label{eq: 28 new}
\int_\mathcal{E} e_x(r) \, T(\mu)(dr) &= \int_\mathcal{E} T(e_x(r)) \, \mu(dr) \nonumber \\
&= \int_\mathcal{E} \big(e^\alpha + 1 - e_x(r)\big) \, \mu(dr) \\
&= (e^\alpha + 1)\mu(\mathcal{E}) - \int_\mathcal{E} e_x(r) \, \mu(dr). \nonumber
\end{align}
Substituting \eqref{eq: 28 new} back into the expression for $\mu^{(s)}$, we observe that the terms involving the integral of $e_x(r)$ with respect to $\mu$ cancel precisely:
\begin{align*}
\int_\mathcal{E} e_x(r) \, \mu^{(s)}(dr) &= \frac{1}{(e^\alpha + 1)\mu(\mathcal{E})} \left[ \int_\mathcal{E} e_x(r) \, \mu(dr) + (e^\alpha + 1)\mu(\mathcal{E}) - \int_\mathcal{E} e_x(r) \, \mu(dr) \right] \\
&= \frac{(e^\alpha + 1)\mu(\mathcal{E})}{(e^\alpha + 1)\mu(\mathcal{E})} = 1.
\end{align*}
This confirms that the normalization constraint for $\mu^{(s)}$ is satisfied.\\
\\
Next, we establish the symmetry of $\mu^{(s)}$ by showing it is invariant under the transformation $T$. Exploiting the linearity of the operator and the fact that $T$ is an involution (i.e., $T \circ T = \text{id}$), we obtain:
\begin{align*}
T(\mu^{(s)}) &= \frac{1}{(e^\alpha + 1)\mu(\mathcal{E})} \big[ T(\mu) + T(T(\mu)) \big] \\
&= \frac{1}{(e^\alpha + 1)\mu(\mathcal{E})} \big[ T(\mu) + \mu \big] \\
&= \mu^{(s)},
\end{align*}
which completes the proof.}

    \item {We now turn to the second statement. Suppose $\mu$ is symmetric, such that $T(\mu) = \mu$. From this symmetry property, it follows that:
\begin{equation*}
\int_\mathcal{E} e_x(r) \, T(\mu)(dr) = \int_\mathcal{E} e_x(r) \, \mu(dr).
\end{equation*}
On the other hand, by \eqref{eq: 28 new}, we have:
\begin{equation*}
\int_\mathcal{E} e_x(r) \, T(\mu)(dr) = (e^\alpha + 1) \mu(\mathcal{E}) - \int_\mathcal{E} e_x(r) \, \mu(dr).
\end{equation*}
Equating these two expressions yields:
\begin{equation*}
2 \int_\mathcal{E} e_x(r) \, \mu(dr) = (e^\alpha + 1) \mu(\mathcal{E}).
\end{equation*}
By the normalization constraint $\int_\mathcal{E} e_x(r) \, \mu(dr) = 1$, the above identity directly implies that the total mass of the measure must be:
\begin{equation*}
\mu(\mathcal{E}) = \frac{2}{e^\alpha + 1}.
\end{equation*}
Finally, we substitute this value and the symmetry condition $T(\mu) = \mu$ into the definition of $\mu^{(s)}$ given in \eqref{eq: mu(s)}:
\begin{align*}
\mu^{(s)} &= \frac{1}{(e^\alpha + 1)\mu(\mathcal{E})} \big( \mu + T(\mu) \big) \\
&= \frac{1}{(e^\alpha + 1) \frac{2}{e^\alpha + 1}} \big( \mu + \mu \big) \\
&= \frac{1}{2} (2\mu) = \mu,
\end{align*}
which confirms that $\mu^{(s)} = \mu$ in the symmetric case.}
\end{enumerate}
\end{proof}

\subsection{Proof of Theorem \ref{thm: Fisher sym}}
\begin{proof}
Let us split this proof into four steps. \\
\\
 In Step 1, we show that if a measure $\mu$ is symmetric and normalized, then 
\begin{align*}
\mathcal{I}_{\theta_0}(q^{(\mu)} \circ \mathcal{P}) &= \int_{\mathcal{E}} i^{(s)} (r) \mu(dr) \\
& = (e^\alpha - 1)^2\frac{e^\alpha + 1}{2} \int_\mathcal{E} \frac{\left(\int_{F_r^+} s_{\theta_0}(x) p_{\theta_0}(x) \, dx\right)^2}{\left(1 + (e^\alpha - 1)\int_{F_r^+} p_{\theta_0}(x) \, dx\right) \left(e^\alpha - (e^\alpha - 1) \int_{F_r^+} p_{\theta_0}(x) \, dx\right)} \mu(dr).
\end{align*}
 In Step 2, we consider $\bar{\mu}$, the symmetric (by assumption) solution of the optimization problem, and we show that for any $H \in L^\infty(\mathcal{E})$ such that $\int_\mathcal{E} H(r) \bar{\mu}(dr) = 0$, it also holds that $\int_\mathcal{E} i^{(s)}(r) H(r) \bar{\mu}(dr) = 0$.\\
\\
In Step 3, we use this fact to conclude that $i^{(s)}(r)$ is $\bar{\mu}$-a.e. constant, or equivalently,
  \[
{\modar i^{(s)}(r)=i^*,~\bar{\mu}(dr) \text{ a.e.,\quad with \quad} {i^*}=\bar{\mu}(\mathcal{E})^{-1}\int_\mathcal{E} i^{(s)}(r) \bar{\mu}(dr).}
  \] 
  Putting Steps 1 and 3 together, we find that
  \begin{equation}\label{eq: fisher 1 steps1-3}
    \mathcal{J}_{\theta_0}^{\max, \alpha} = \mathcal{I}_{\theta_0}(q^{(\bar{\mu})} \circ \mathcal{P}) = \int_{\mathcal{E}} i^{(s)}(r) \bar{\mu}(dr) = i^* \bar{\mu}(\mathcal{E}).
  \end{equation}
 In Step 4, we consider $\max_{r \in B} i^{(s)}(r)$ and show that the maximum is attained on $\mathcal{E}$, so that its value coincides with the constant $i^*$. We can deduce, thanks to \eqref{eq: fisher 1 steps1-3} and Point 2 of Proposition~\ref{prop: symmetrisée}, that
  \[
  \mathcal{J}_{\theta_0}^{\max, \alpha} = \left( \max_{r \in B} i^{(s)}(r) \right) \frac{2}{e^\alpha + 1},
  \]
  which is the desired result.

\medskip

\noindent \textbf{Step 1.} Recall that, from \eqref{eq : Fisher ctn general_preli}, we have
\begin{equation}\label{eq: fish start sym}
  \mathcal{I}_{\theta_0}(q^{(\mu)} \circ \mathcal{P}) = \int_\mathcal{E} i(r) \mu(dr),
\end{equation}
with
\[
i(r) = \frac{\left( \int_\mathcal{X} s_{\theta_0}(x) r(x) p_{\theta_0}(x) \, dx \right)^2}{\int_\mathcal{X} r(x) p_{\theta_0}(x) \, dx}.
\]
Since $\mu$ is symmetric, we have (from Point 2 of Proposition~\ref{prop: symmetrisée}):
\[
\mu = \mu^{(s)} = \frac{1}{2}[\mu + T(\mu)].
\]
Plugging this into \eqref{eq: fish start sym} gives
\begin{equation}\label{eq: fish step 1 3}
  \mathcal{I}_{\theta_0}(q^{(\mu)} \circ \mathcal{P}) = \int_\mathcal{E} i(r) \, \frac{1}{2}[\mu + T(\mu)](dr) = \frac{1}{2} \int_\mathcal{E} [i(r) + i(T(r))] \mu(dr),
\end{equation}
using the definition of image measure. Recalling the definition
\[
i^{(s)}(r) := \frac{1}{2}[i(r) + i(T(r))],
\]
we have established the first equality of Step 1.

To prove the second part, we compute $i(T(r))$ in detail. Using the definitions of the symmetry operator $T$ (as in \eqref{eq: def T sym}) and $i(r)$ (as in \eqref{eq : def i r ctn}), we find:
\begin{align}{\label{eq: i(T(r))}}
i(T(r)) &= i(e^\alpha + 1 - r) = \frac{\left( \int_\mathcal{X} s_{\theta_0}(x) (e^\alpha + 1 - r(x)) p_{\theta_0}(x) \, dx \right)^2}{\int_\mathcal{X} (e^\alpha + 1 - r(x)) p_{\theta_0}(x) \, dx} \\
&= \frac{\left( \int_\mathcal{X} s_{\theta_0}(x) r(x) p_{\theta_0}(x) \, dx \right)^2}{e^\alpha + 1 - \int_\mathcal{X} r(x) p_{\theta_0}(x) \, dx}, \nonumber
\end{align}
where we used the facts that 
\begin{equation}{\label{eq: 40.5}}
 \int_\mathcal{X} s_{\theta_0}(x) p_{\theta_0}(x) \, dx = 0   
\end{equation}
and 
\begin{equation}{\label{eq: 40.75}}
 \int_\mathcal{X} p_{\theta_0}(x) \, dx = 1.   
\end{equation}
Comparing this with the expression for $i(r)$, we see that the numerator remains unchanged, while the denominator is transformed under $T$. Plugging this into \eqref{eq: fish step 1 3}, we get:
\begin{align}{\label{eq: 40.9}}
\mathcal{I}_{\theta_0}(q^{(\mu)} \circ \mathcal{P}) &= \frac{1}{2} \int_\mathcal{E} \left( \int_\mathcal{X} s_{\theta_0}(x) r(x) p_{\theta_0}(x) \, dx \right)^2 \left( \frac{1}{\int_\mathcal{X} r(x) p_{\theta_0}(x) \, dx} + \frac{1}{e^\alpha + 1 - \int_\mathcal{X} r(x) p_{\theta_0}(x) \, dx} \right) \mu(dr) \nonumber \\
&= \frac{1}{2}(e^\alpha + 1) \int_\mathcal{E} \frac{\left( \int_\mathcal{X} s_{\theta_0}(x) r(x) p_{\theta_0}(x) \, dx \right)^2}{\left( \int_\mathcal{X} r(x) p_{\theta_0}(x) \, dx \right) \left( e^\alpha + 1 - \int_\mathcal{X} r(x) p_{\theta_0}(x) \, dx \right)} \mu(dr).
\end{align}
Finally, replacing $r(x) = 1 + (e^\alpha - 1) \indi{F_r^+}$ and using \eqref{eq: 40.5} and \eqref{eq: 40.75} concludes Step 1. \\
\\
\textbf{Step 2.} Let $H \in L^\infty(\mathcal{E})$ be such that $\int_\mathcal{E} H(r)\, \bar{\mu}(dr) = 0$, and define $H^s := H + H \circ T$. Observe that, since $\bar{\mu}$ is symmetric by hypothesis, we can apply Point 2 of Proposition~\ref{prop: symmetrisée} to deduce $T(\bar{\mu}) = \bar{\mu}$. Hence,
\begin{align*}
\int_\mathcal{E} H^s(r)\, \bar{\mu}(dr) & = \int_\mathcal{E} H(r)\, \bar{\mu}(dr) + \int_\mathcal{E} H(T(r))\, \bar{\mu}(dr) \\
&= \int_\mathcal{E} H(r)\, \bar{\mu}(dr) + \int_\mathcal{E} H(r)\, T(\bar{\mu})(dr) \\
&= 2 \int_\mathcal{E} H(r)\, \bar{\mu}(dr) = 0.
\end{align*}
We now use $\epsilon H^s$ to perturb $\bar{\mu}$ for some {\modar $\epsilon\in \mathbb{R}$.} Define the perturbed measure by
\[
\mu^{\epsilon, H}(dr) := (1 + \epsilon H^s(r))\, \bar{\mu}(dr).
\]
We show that $\mu^{\epsilon, H}$ is a normalized probability measure {\modar for all $\abs*{\epsilon}< \frac{1}{2}\norm{H}_{{L}^\infty(\mathcal{E})}$. 
It is clear that $\mu^{\epsilon, H}$ is a non-negative measure.}
 Using that $\bar{\mu}$ is normalized, we compute
\begin{align}
\int_{\mathcal{E}} e_x(r)\, \mu^{\epsilon, H}(dr) &= \int_{\mathcal{E}} e_x(r)\, \bar{\mu}(dr) + \epsilon \int_{\mathcal{E}} e_x(r) H^s(r)\, \bar{\mu}(dr) \notag \\
&= 1 + \epsilon \int_{\mathcal{E}} e_x(r) H(r)\, \bar{\mu}(dr) + \epsilon \int_{\mathcal{E}} e_x(r) H(T(r))\, \bar{\mu}(dr). \label{eq: norm perturb 4}
\end{align}
For the last term, using that $T \circ T = \mathrm{id}$ and that $\bar{\mu}$ is symmetric, it equals
\[
\int_{\mathcal{E}} e_x(T(r)) H(r)\, T(\bar{\mu})(dr) = \int_{\mathcal{E}} e_x(e^\alpha + 1 - r) H(r)\, \bar{\mu}(dr) = (e^\alpha + 1) \int_{\mathcal{E}} H(r)\, \bar{\mu}(dr) - \int_{\mathcal{E}} e_x(r) H(r)\, \bar{\mu}(dr).
\]
Plugging this into~\eqref{eq: norm perturb 4}, and recalling that $\int_\mathcal{E} H(r)\, \bar{\mu}(dr) = 0$, we get
\[
\int_{\mathcal{E}} e_x(r)\, \mu^{\epsilon, H}(dr) = 1 + \epsilon \int_{\mathcal{E}} e_x(r) H(r)\, \bar{\mu}(dr) - \epsilon \int_{\mathcal{E}} e_x(r) H(r)\, \bar{\mu}(dr) = 1.
\]
Thus, $\mu^{\epsilon, H}$ is a probability measure satisfying \eqref{eq : cond norm mu continue}, as required. Since $\bar{\mu}$ solves the constrained optimization problem, we must have
\[
\int_\mathcal{E} i(r)\, \bar{\mu}(dr) \geq \int_\mathcal{E} i(r)\, \mu^{\epsilon, H}(dr) = \int_\mathcal{E} i(r)\, \bar{\mu}(dr) + \epsilon \int_\mathcal{E} i(r) H^s(r)\, \bar{\mu}(dr).
\]
Since this holds for all {\modar $\epsilon $ in a neighborhood of $0$,} it follows that
\[
\int_\mathcal{E} i(r) H^s(r)\, \bar{\mu}(dr) = 0.
\]
We now show that this implies
\[
\int_\mathcal{E} i^{(s)}(r) H(r)\, \bar{\mu}(dr) = 0.
\]
Indeed, using the definition of $H^s$ and symmetry of $\bar{\mu}$, we compute:
\begin{align*}
0 &= \int_\mathcal{E} i(r) H^s(r)\, \bar{\mu}(dr) = \int_\mathcal{E} i(r) \left[ H(r) + H(T(r)) \right]\, \bar{\mu}(dr) \\
&= \int_\mathcal{E} i(r) H(r)\, \bar{\mu}(dr) + \int_\mathcal{E} i(T(r)) H(r)\, T(\bar{\mu})(dr) = \int_\mathcal{E} \left[ i(r) + i(T(r)) \right] H(r)\, \bar{\mu}(dr) \\
&= 2 \int_\mathcal{E} i^{(s)}(r) H(r)\, \bar{\mu}(dr),
\end{align*}
which concludes the proof of Step 2.\\
\\
\textbf{Step 3.} We now use the result from Step 2 to show that $i^{(s)}$ is $\bar{\mu}$-a.e. constant. Let $G \in L^\infty(\mathcal{E})$ and let us define its centered version, on which we can apply Step 2:
\[
\bar{G} := G - \frac{1}{\bar{\mu}(\mathcal{E})} \int_\mathcal{E} G(r)\, \bar{\mu}(dr),
\]
so that 
\begin{equation}{\label{eq: 41.5}}
 \int_\mathcal{E} \bar{G}(r)\, \bar{\mu}(dr) = 0.  
\end{equation}
We can then use Step 2, which implies:
\begin{equation} \label{eq: conseq step 2 5}
\int_\mathcal{E} i^{(s)}(r) \bar{G}(r)\, \bar{\mu}(dr) = 0.
\end{equation}
Now define the centered version of $i^{(s)}$:
\[
\bar{i}^{(s)} := i^{(s)} - \frac{1}{\bar{\mu}(\mathcal{E})} \int_\mathcal{E} i^{(s)}(r)\, \bar{\mu}(dr).
\]
Using \eqref{eq: 41.5} and \eqref{eq: conseq step 2 5}, we deduce
\[
\int_\mathcal{E} \bar{i}^{(s)}(r)\, \bar{G}(r)\, \bar{\mu}(dr) = 0,
\]
and by construction $\int_\mathcal{E} \bar{i}^{(s)}(r)\, \bar{\mu}(dr) = 0$. Then, expanding $\bar{G}$, we obtain
\[
\int_\mathcal{E} \bar{i}^{(s)}(r) G(r)\, \bar{\mu}(dr) = 0.
\]
Since this holds for all $G \in L^\infty(\mathcal{E})$, it follows that $\bar{i}^{(s)} = 0$ $\bar{\mu}$-a.e., i.e.,
\[
i^{(s)} = \frac{1}{\bar{\mu}(\mathcal{E})} \int_\mathcal{E} i^{(s)}(r)\, \bar{\mu}(dr)=: i^* \quad \text{$\bar{\mu}$-a.e.}
\]
This concludes the proof of Step 3. \\
\\
\textbf{Step 4.} As anticipated in \eqref{eq: fisher 1 steps1-3}, we have so far established that
\[
\mathcal{J}_{\theta_0}^{\max, \alpha} = \mathcal{I}_{\theta_0}(q^{(\bar{\mu})} \circ \mathcal{P}) = \int_\mathcal{E} i^{(s)}(r) \, \bar{\mu}(dr) = i^{*} \bar{\mu} (\mathcal{E}).
\]
In this final step, we aim to prove that
\[
i^{*} = \max_{r \in B} i^{(s)}(r).
\]
Observe that the set \( B \), defined as in \eqref{eq: set B}, is clearly convex. Moreover, it is compact with respect to the weak-\(\star\) topology, thanks to Lemma 3 in \cite{amorino2025factorization} (see Section 6.1 of the same manuscript for an introduction to the weak-\(\star\) topology, particularly in the context of its application to local differential privacy). Consequently, since the mapping \( r \mapsto i^{(s)}(r) \) is continuous, the supremum
\[
\sup_{r \in B} i^{(s)}(r)
\]
is attained at some point \( r_{\max} \in B \), i.e.,
\[
\sup_{r \in B} i^{(s)}(r) = i^{(s)}(r_{\max}).
\]

We now claim that, due to the convexity of \( r \mapsto i^{(s)}(r) \), this maximum is attained on the set of extreme points \( \mathcal{E} \); that is, \( r_{\max} \in \mathcal{E} \). From this, it would follow that
\[
i^{(s)}(r_{\max}) = i^{*},
\]
i.e., the (constant on $\mathcal{E}$) value introduced in the previous step, thereby completing the proof of the theorem. 

To justify the claim, observe that since \( r_{\max} \in B \), Choquet’s theorem (see Theorem 5 in \cite{amorino2025factorization}) ensures the existence of a non-negative Radon measure \( \nu \) supported on \( \mathcal{E} \) such that
\[
r_{\max} = \int_\mathcal{E} r \, \nu(dr).
\]
By convexity of \( i^{(s)} \), we then have
\[
i^{(s)}(r_{\max}) \le \int_\mathcal{E} i^{(s)}(r) \, \nu(dr),
\]
which implies
\[
\int_\mathcal{E} \left[i^{(s)}(r) - i^{(s)}(r_{\max})\right] \nu(dr) \ge 0.
\]
On the other hand, from the definition of \( r_{\max} \), we know that \( i^{(s)}(r) \le i^{(s)}(r_{\max}) \) for all \( r \in B \), and hence for all \( r \in \mathcal{E} \). Therefore, the integrand is non-positive, and the only way the integral can be non-negative is if the integrand vanishes \(\nu\)-almost everywhere:
\[
i^{(s)}(r) = i^{(s)}(r_{\max}) \quad \text{for all } r \in \operatorname{supp}(\nu) \subseteq \mathcal{E}.
\]
Thus, the maximum is achieved on \( \mathcal{E} \), as claimed. This concludes the proof of the fourth and final step, and therefore of the theorem.

\end{proof}

\subsection{Proof of Proposition \ref{p: decomp fisher sym}}
\begin{proof}
We start by decomposing the measure \( \mu \) into its symmetric and asymmetric parts:
\begin{equation}\label{eq: start decomp 1}
\mathcal{I}_{\theta_0}(q^{(\mu)} \circ \mathcal{P}) = \int_{\mathcal{E}} i(r) \mu(dr) = \int_{\mathcal{E}} i(r) \mu^{(s)}(dr) + \int_{\mathcal{E}} i(r) \mu^{(as)}(dr).
\end{equation}
Since \( \mu^{(s)} \) is symmetric, the first term equals \( \int_{\mathcal{E}} i^{(s)}(r) \mu(dr) \), as shown in the proof of Theorem~\ref{thm: Fisher sym} (cf.~Equation~\eqref{eq: fish step 1 3}). For the asymmetric part, using the definition of \( \mu^{(s)} \) in \eqref{eq: mu(s)}, we write
\[
\mu^{(as)} = \mu - \mu^{(s)} = \frac{1}{(1 + e^\alpha) \mu(\mathcal{E})} \left( ((1 + e^\alpha) \mu(\mathcal{E}) - 1)\mu - T(\mu) \right).
\]
Substituting into \eqref{eq: start decomp 1} and using the definition of pushforward measure, we get
\begin{equation}\label{eq: asym decomp 2}
\int_{\mathcal{E}} i(r) \mu^{(as)}(dr) = \frac{1}{(1 + e^\alpha) \mu(\mathcal{E})} \int_{\mathcal{E}} \left[i(r)((1 + e^\alpha) \mu(\mathcal{E}) - 1) - i(T(r)) \right] \mu(dr).
\end{equation}
Using the definitions of \( i(r) \) and \( i(T(r)) \) in \eqref{eq : def i r ctn} and \eqref{eq: i(T(r))}, we compute:
\begin{align*}
& i(r)((1 + e^\alpha) \mu(\mathcal{E}) - 1) - i(T(r)) \\
&= \left( \int_\mathcal{X} s_{\theta_0}(x) r(x) p_{\theta_0}(x) dx \right)^2 \left( \frac{(1 + e^\alpha) \mu(\mathcal{E}) - 1}{\int_\mathcal{X} r(x) p_{\theta_0}(x) dx} - \frac{1}{e^\alpha + 1 - \int_\mathcal{X} r(x) p_{\theta_0}(x) dx} \right) \\
&= \frac{\left( \int_\mathcal{X} s_{\theta_0}(x) r(x) p_{\theta_0}(x) dx \right)^2 (1 + e^\alpha)\mu(\mathcal{E})}{\left( \int_\mathcal{X} r(x) p_{\theta_0}(x) dx \right)(e^\alpha + 1 - \int_\mathcal{X} r(x) p_{\theta_0}(x) dx)} \left( e^\alpha + 1 - \frac{1}{\mu(\mathcal{E})} - \int_\mathcal{X} r(x) p_{\theta_0}(x) dx \right).
\end{align*}
Recall that, from \eqref{eq: 40.9},
\begin{align*}
i^{(s)}(r) = \frac{e^\alpha + 1}{2} \frac{\left( \int_\mathcal{X} s_{\theta_0}(x) r(x) p_{\theta_0}(x) \, dx \right)^2}{\left( \int_\mathcal{X} r(x) p_{\theta_0}(x) \, dx \right) \left( e^\alpha + 1 - \int_\mathcal{X} r(x) p_{\theta_0}(x) \, dx \right)}.
\end{align*}
Recognizing its expression in the equation above, this implies
\[
\int_{\mathcal{E}} i(r) \mu^{(as)}(dr) = 2 \mu(\mathcal{E}) \int_{\mathcal{E}} \left[ e^\alpha + 1 - \frac{1}{\mu(\mathcal{E})} - \int_{\mathcal{X}} r(x) p_{\theta_0}(x) dx \right] i^{(s)}(r) \mu(dr).
\]
Now, since \( r(x) = 1 + (e^\alpha - 1)\mathbf{1}_{F_r^+}(x) \), we find
\[
\int_\mathcal{X} r(x) p_{\theta_0}(x) dx = 1 + (e^\alpha - 1)\int_{F_r^+} p_{\theta_0}(x) dx,
\]
so that
\begin{align*}
e^\alpha + 1 - \frac{1}{\mu(\mathcal{E})} - \int_\mathcal{X} r(x) p_{\theta_0}(x) dx &= e^\alpha - \frac{1}{\mu(\mathcal{E})} - (e^\alpha - 1)\int_{F_r^+} p_{\theta_0}(x) dx \\
&= \left( \frac{e^\alpha + 1}{2} - \frac{1}{\mu(\mathcal{E})} \right) - (e^\alpha - 1)\left( \int_{F_r^+} p_{\theta_0}(x) dx - \frac{1}{2} \right).    
\end{align*}
It follows that
\begin{align*}
\mathcal{I}_{\theta_0}(q^{(\mu)} \circ \mathcal{P})& = \int_{\mathcal{E}} i^{(s)}(r) \mu(dr) + 2 \mu(\mathcal{E}) \left( \frac{e^\alpha + 1}{2} - \frac{1}{\mu(\mathcal{E})} \right) \int_{\mathcal{E}} i^{(s)}(r) \mu(dr) \\
& - 2 \mu(\mathcal{E})(e^\alpha - 1)\left( \int_{F_r^+} p_{\theta_0}(x) dx - \frac{1}{2} \right) \int_{\mathcal{E}} i^{(s)}(r) \mu(dr).
\end{align*}
Recalling that \( i^{(as)}(r) = i^{(s)}(r) \left( \int_{F_r^+} p_{\theta_0}(x) dx - \frac{1}{2} \right) \), this gives the desired decomposition.

Let us now prove that \( i^{(as)}(T(r)) = -i^{(as)}(r) \). Since \( i^{(s)}(T(r)) = i^{(s)}(r) \) and \( T \) exchanges \( F_r^+ \) with its complement, we have
\[
i^{(as)}(T(r)) = i^{(s)}(r) \left( \int_{(F_r^+)^c} p_{\theta_0}(x) dx - \frac{1}{2} \right) = i^{(s)}(r) \left(\frac{1}{2} - \int_{F_r^+} p_{\theta_0}(x) dx \right) = - i^{(as)}(r).
\]

To conclude, in the regime \( \alpha \to 0 \), Equation \eqref{eq: i(s)r} yields:
{\modar
	\begin{align*}
 i^{(s)}(r) &= \alpha^2 \left( \int_{F_r^+} s_{\theta_0}(x) p_{\theta_0}(x) dx \right)^2 + O(\alpha^3),
\\ 
\quad i^{(as)}(r) &=\alpha^2 \left( \int_{F_r^+} s_{\theta_0}(x) p_{\theta_0}(x) dx \right)^2 \left( \int_{F_r^+} p_{\theta_0}(x) dx - \frac{1}{2} \right)+O(\alpha^3).
\end{align*}}
Moreover, from the bound on \( \mu(\mathcal{E}) \) in \eqref{eq: bound mu E}, we find
\[{\modar
\frac{e^\alpha + 1}{2} - \frac{1}{\mu(\mathcal{E})} 
=\frac{e^\alpha + 1}{2} - \frac{1}{1+O(\alpha)}=O(\alpha), }
\]
which, combined with the decomposition above, yields the result.
\end{proof}

\subsection{Proof of Lemma \ref{l : second step staircase}}

\begin{proof}	
	We define $x_{\text{min},c}=\Xi^{-1}(c/2)$ and  
	$x_{\text{max},c}=\Xi^{-1}(1-c/2)$.
	Remark that $g_c$ is left continuous, with a discontinuity at $x^*_{g,c}$ as
	$\lim_{x {\downarrow} x^*_{g,c}} g_c(x)=\Xi^{-1}(c/2)=x_{\text{min},c}>-\infty$. 
	The function $d_c$ is right continuous and $\lim_{x \uparrow x^{*}_{d,c}}d_c(x)=x_{\text{max},c}.$
	Let us represent the variation of $g_c$ by this table:
	
	\begin{tikzpicture}
		\tkzTabInit[espcl=4,deltacl=0.7]{$x$ / 1, $g_c(x)$ / 1.7}{$-\infty$, $x^*_{g,c}=\Xi^{-1}(c)$, , $+\infty$}
		\tkzTabVar{-/{$-\infty$},-D-/{$-\infty$}
			/{$x_{\text{min},c}$}
			,R/ ,+/ $x_{\text{max},c}$}
	\tkzTabVal{2}{4}{0.5}{}{	$\Xi^{-1}(-\frac{c}{2}+\Xi(\cdot)) $}
	\end{tikzpicture}

\noindent	
From the definition of the inverse of the left continuous function $g_c$, we can deduce that
\begin{equation*}
	g^{-1}_c(x):=\begin{cases}
		\Xi^{-1}(c) & \text{ for $x\in (-\infty, x_{\text{min},c}]$,}
		\\
		\Xi^{-1}\left(\frac{c}{2}+\Xi(x)\right) &\text{ for $x\in (x_{\text{min},c},x_{\text{max},c})$,}
		\\
		+\infty & \text{ for $x\in [x_{\text{max},c},\infty)$.}
	\end{cases}
\end{equation*}	
In the same way, we can prove	
\begin{equation*}
	d_c^{-1}(x):=\begin{cases}
		-\infty &\text{ for $x \in (-\infty, x_{\text{min},c}]$,}
		\\		
		\Xi^{-1}\left(-\frac{c}{2}+\Xi(x)\right) &\text{ for $x\in (x_{\text{min},c},x_{\text{max},c})$,}
		\\
		\Xi^{-1}(1-c) &\text{ for $x\in [x_{\text{max},c},\infty )$.}
	\end{cases}
\end{equation*}
Now, we can compute the value of
\begin{align*}
	 \int_{d_c^{-1}(x)}^{g_c^{-1}(x)}\nu(x_0) \, dx_0&=
	 \Xi(g_c^{-1}(x)) - \Xi(d_c^{-1}(x))
	 \\&=\begin{cases*}
	 \Xi(\Xi^{-1}(c)) &\text{ for $x \le x_{\text{min},c}$,}
	 \\
	 \Xi( \Xi^{-1}\left(\frac{c}{2}+\Xi(x)\right) ) -
	 \Xi( \Xi^{-1}\left(-\frac{c}{2}+\Xi(x)\right) ) 
	 &\text{ for $x\in (x_{\text{min},c},x_{\text{max},c})$,}
	 \\	
	 1-\Xi(\Xi^{-1}(1-c)) &\text{ for $x\ge x_{\text{max},c}$,}	
	 \end{cases*}
	 \\
	 &=c.
\end{align*}		
\end{proof}

\subsection{Proof of Proposition \ref{prop: kernel}}
\begin{proof}
	The proof consists in explicitly computing the Fisher information associated with the proposed privacy mechanism.
	 According to Equation \eqref{eq : Fisher ctn general_preli}, and {\modar introducing the notation 
	 	$\dot{p}_{\theta_0}(x)= s_{\theta_0}(x) p_{\theta_0}(x)$,}
we have:
	\begin{equation*}
	\mathcal{I}_{\theta_0} (q^{({\modar \mu})} \circ \mathcal{P}) = \int_\mathcal{E} \frac{\left( \int_\mathcal{X} \dot{p}_{\theta_0}(x) r(x) \, dx \right)^2}{\tilde{p}_{\theta_0}(r)} \, {\modar \mu(dr)} =  \int_\mathcal{E} i(r)  \, {\modar \mu}(dr),
	\end{equation*}
	where we use the notation \eqref{eq : def i r ctn}. 
	By \eqref{eq : mu_0 mu} and the change of variable formula \eqref{eq: mu kernel}, it gives,
	\begin{equation} \label{eq : I staircase with nu_0}
		\mathcal{I}_{\theta_0} (q^{({\modar \mu})} \circ \mathcal{P}) =
		\int_\mathbb{R}  i(I(x_0)) \frac{{\modar\nu}(x_0)}{1+c(e^\alpha-1)} \, dx_0 .
	\end{equation}	
	From the expression {\modar $r_{x_0=}=1+ (e^\alpha - 1) \indi{[g_c(x_0),d_c(x_0)]}$,} we have 
	\begin{equation} \label{eq : i I x0 explicite}
		i(I(x_0))=i(r_{x_0})=\frac{\left(e^\alpha-1\right)^2  \left( \int_{g(x_0)}^{d(x_0)}\dot{p}_{\theta_0}(x)\,dx\right)^2}
		{1+\left(e^\alpha-1\right) \int_{g(x_0)}^{d(x_0)} p_{\theta_0}(x)\,dx}.
	\end{equation}
As $c \to 0$, we can assume $c<1/2$ and we split the right hand side of \eqref{eq : I staircase with nu_0}
into three parts:
\begin{equation*}
		\mathcal{I}_{\theta_0} (q^{(\mu^0)} \circ \mathcal{P}) =
	\int_{-\infty}^{\Xi^{-1}(c)}  \dots + 
	\int_{\Xi^{-1}(c)}^{\Xi^{-1}(1-c)} \dots +
 	\int_{\Xi^{-1}(1-c)}^\infty \dots
 	= \sum_{l=1}^3 \mathcal{I}^{(l)}.
\end{equation*}	
Now, we establish the following three convergences, as $c \to 0$, $ce^\alpha\to\infty$:
\begin{equation}  \label{eq : a montrer trois cv}
	\mathcal{I}^{(1)} \to 0, \quad 	\mathcal{I}^{(3)} \to 0, \quad 	\mathcal{I}^{(2)} \to \int_\mathbb{R} \frac{\dot{p}_{\theta_0}(x)^2}{p_{\theta_0}(x)}\, dx.
\end{equation} 
First, let us obtain some upper bound on $i(r_{x_0})$.  From \eqref{eq : i I x0 explicite} and using Cauchy-Schwarz inequality, we have
\begin{multline} \label{eq : bound on i r x0}
	i(r_{x_0}) \le \frac { \left(e^\alpha-1\right)  \left( \int_{{\modar g_c}(x_0)}^{{\modar d_c}(x_0)}\dot{p}_{\theta_0}(x)\,dx\right)^2}
	{ \int_{{\modar g_c}(x_0)}^{{\modar d_c}(x_0)} p_{\theta_0}(x)\,dx} 
	\le   \left(e^\alpha-1\right)  \frac{ \left( \int_{{\modar g_c}(x_0)}^{{\modar d_c}(x_0)} \dot{p}_{\theta_0}(x)^2 /  p_{\theta_0}(x) \,dx\right)
	 \left( \int_{{\modar g_c}(x_0)}^{{\modar d_c}(x_0)} p_{\theta_0}(x) \,dx\right) }
	{ \int_{{\modar g_c}(x_0)}^{{\modar d_c}(x_0)} p_{\theta_0}(x)\,dx} 	
\\	= \left(e^\alpha-1\right) \left( \int_{{\modar g_c}(x_0)}^{{\modar d_c}(x_0)} \dot{p}_{\theta_0}(x)^2 /  p_{\theta_0}(x) \,dx\right).
\end{multline}
Now, let us focus on the convergence of  $	I^{(1)}$. Using the previous upper bound, we have
\begin{equation*}
		\mathcal{I}^{(1)} \le \left(e^\alpha-1\right) \int_{-\infty}^{\Xi^{-1}(c)}  
		\left( \int_{{\modar g_c}(x_0)}^{{\modar d_c}(x_0)}  \frac{ \dot{p}_{\theta_0}(x)^2 }{ p_{\theta_0}(x) } \, dx
		\right)
		\frac{{\modar\nu}(x_0)}{1+c(e^\alpha-1)} \,	dx_0.
\end{equation*}
For $x_0 \in (-\infty, \Xi^{-1}(c))$, we have by \eqref{eq : def lc} that
${\modar g_c}(x_0)=-\infty$ and, inserting in the previous upper bound the expression given in \eqref{eq : def rc}
for ${\modar d_c}(x_0)$, 
\begin{equation*}
		\mathcal{I}^{(1)}  \le  \left(e^\alpha-1\right) \int_{-\infty}^{\Xi^{-1}(c)}  
		\left( \int_{-\infty}^{ 	\Xi^{-1}(\frac{c}{2}+\Xi(x_0)) }  \frac{ \dot{p}_{\theta_0}(x)^2 }{ p_{\theta_0}(x) } \,dx \right)
		\frac{{\modar \nu}(x_0)}{1+c(e^\alpha-1)} \,dx_0 .
\end{equation*}
By the change of variable $u_0=\Xi(x_0)$ and recalling $\Xi'(x_0)={\modar \nu}(x_0)$, we have
\begin{align*}
	\mathcal{I}^{(1)}& \le  \left(e^\alpha-1\right) \int_{0}^{c}  
	\left( \int_{-\infty}^{ 	\Xi^{-1}(\frac{c}{2}+u_0) }  \frac{ \dot{p}_{\theta_0}(x)^2 }{ p_{\theta_0}(x) } \,dx \right)
	\frac{du_0}{1+c(e^\alpha-1)} 
	\\
	& \le  \left(e^\alpha-1\right) \int_{0}^{c}  
	\left( \int_{-\infty}^{ 	\Xi^{-1}(\frac{3c}{2}) }  \frac{ \dot{p}_{\theta_0}(x)^2 }{ p_{\theta_0}(x) } \,dx \right)
	\frac{du_0}{1+c(e^\alpha-1)} 
	\\
	& \le  \frac{c  \left(e^\alpha-1\right) }{1+c(e^\alpha-1)}
	\int_{-\infty}^{ \Xi^{-1}(\frac{3c}{2}) }  
	\frac{ \dot{p}_{\theta_0}(x)^2 }{ p_{\theta_0}(x) }\,dx \le 	\int_{-\infty}^{ \Xi^{-1}(\frac{3c}{2}) }  
	\frac{ \dot{p}_{\theta_0}(x)^2 }{ p_{\theta_0}(x) }\,dx
\end{align*}
From the finiteness of the Fisher information of the model, and $\Xi^{-1}(\frac{3c}{2}) \xrightarrow{c\to0} -\infty$, we deduce
that  $	\mathcal{I}^{(1)} \xrightarrow{c\to0} 0$.

We prove in the same way that  $\mathcal{I}^{(3)} \xrightarrow{c\to0} 0$. 

Let us focus on the convergence of $\mathcal{I}^{(2)}$. Using the expressions for ${\modar g_c}(x_0)$ and ${\modar d_c}(x_0)$ given in
\eqref{eq : def lc}--\eqref{eq : def rc} when $x_0 \in (\Xi^{-1}(c),\Xi^{-1}(1-c))$ together with the change of variable $u_0=\Xi(x_0)$, we have
\begin{align*}
\mathcal{I}^{(2)}
&=		\int_{\Xi^{-1}(c)}^{\Xi^{-1}(1-c)} 
\frac{\left(e^\alpha-1\right)^2  \left( \int_{{\modar g_c}(x_0)}^{{\modar d_c}(x_0)}\dot{p}_{\theta_0}(x)\,dx\right)^2}
{1+\left(e^\alpha-1\right) \int_{{\modar g_c}(x_0)}^{{\modar d_c}(x_0)} p_{\theta_0}(x)\,dx}
	\frac{{\modar\nu}(x_0)\,dx_0}{1+c(e^\alpha-1)} 
\\
&=\frac{\left(e^\alpha-1\right)^2}{1+c(e^\alpha-1)}\int_c^{1-c}  
\frac{ \left(\int_{\Xi^{-1}(-c/2+u_0)}^{\Xi^{-1}(c/2+u_0)} \dot{p}_{\theta_0}(x)\,dx \right)^2 
}{1 + (e^\alpha-1) \int_{\Xi^{-1}(-c/2+u_0)}^{\Xi^{-1}(c/2+u_0)} 
p_{\theta_0}(x)\,dx   
}
\,du_0 .
\end{align*}
By making the change of variable $w=\Xi(x)$ in the inner integrals, we can write
\begin{equation*}
	\mathcal{I}^{(2)}
=\int_0^1 \hat{i}_{c,\alpha}(u_0) \,du_0
\end{equation*}
with 
\begin{equation} \label{eq : def hat i c alpha}
	\hat{i}_{c,\alpha}(u_0)= \indi{(c,1-c)}(u_0) \frac{\left(e^\alpha-1\right)^2}{1+c(e^\alpha-1)}
	\frac{ \left(\int_{-c/2+u_0}^{c/2+u_0} \frac{\dot{p}_{\theta_0}(\Xi^{-1}(w))}{{\modar \nu}(\Xi^{-1}(w))} \,dw \right)^2 
	}{1 + (e^\alpha-1) \int_{-c/2+u_0}^{c/2+u_0} \frac{p_{\theta_0}(\Xi^{-1}(w))}{{\modar \nu}(\Xi^{-1}(w))} \,dw  }   .
\end{equation} 
Now, the almost everywhere convergence of $\hat{i}_{c,\alpha}$ and the uniform integrability given in Lemma \ref{l: technical cv i} below 
yield the convergence of the integral 
\begin{equation*}
\int_0^1 \hat{i}_{c,\alpha}(u_0) du_0 \xrightarrow{c\to0,~ce^\alpha\to\infty} \int_0^1 \hat{i}_*(u_0) \,du_0
=\int_0^1
 \frac{\dot{p}_{\theta_0}(\Xi^{-1}(u_0))^2}{p_{\theta_0}(\Xi^{-1}(u_0)){\modar \nu}(\Xi^{-1}(u_0))} \,du_0
 =
 \int_\mathbb{R}\frac{\dot{p}_{\theta_0}(x_0)^2}{p_{\theta_0}(x_0)}\, dx_0,
\end{equation*}
where in the final step 
we have set $x_0=\Xi^{-1}(u_0)$.
\end{proof}
\begin{lemma}\label{l: technical cv i}
1) We have for almost all $u_0\in (0,1)$,
\begin{equation}\label{eq : cv hat i to hat i zero}
\hat{i}_{c,\alpha}(u_0) \xrightarrow{c\to0,~ce^\alpha\to\infty}\hat{i}_*(u_0):=
 \frac{\dot{p}_{\theta_0}(\Xi^{-1}(u_0))^2}{p_{\theta_0}(\Xi^{-1}(u_0)){\modar \nu}(\Xi^{-1}(u_0))}.
\end{equation}	
2) Moreover, the family of functions $(\hat{i}_{c,\alpha})_{c\in(0,1/2), \alpha>0}$ is uniformly integrable on $[0,1]$.
\end{lemma}
\begin{proof}
1) We recall that the functions $x\mapsto {p}_{\theta_0}(x)/\nu_0(x)$ and   
	$x\mapsto \dot{p}_{\theta_0}(x)/\nu_0(x)$  are continuous at almost every point $x\in\mathbb{R}$  and that $u\mapsto\Xi^{-1}(u)$ is a $\mathcal{C}^1$ diffeomorphism. It entails that the functions $u \mapsto  {p}_{\theta_0}(\Xi^{-1}(u))/\nu_0(\Xi^{-1}(u))$ and 
	$u \mapsto  \dot{p}_{\theta_0}(\Xi^{-1}(u))/\nu_0(\Xi^{-1}(u))$ are continuous at almost every $u_0\in(0,1)$. For $u_0\in(0,1)$ where these functions are continuous, we have that as $c\to 0$, $ce^\alpha \to \infty$,
\begin{equation*}
\hat{i}_{c,\alpha}(u_0) =
\frac{e^{2\alpha}  \left(c	\dot{p}_{\theta_0}(\Xi^{-1}(u_0))/\nu_0(\Xi^{-1}(u_0))\right)^2 + o(e^{2\alpha}c^2)}
	{ ce^{2\alpha} c {p}_{\theta_0}(\Xi^{-1}(u_0))/\nu_0(\Xi^{-1}(u_0)) + o(e^{2\alpha}c^2) } 	
	\to \hat{i}_*(u_0).
\end{equation*}
	
2) We prove the uniform integrability of the family, which is a key step in the proof of Proposition \ref{prop: kernel} to justify that one can integrate in $u_0$ the convergence	
\eqref{eq : cv hat i to hat i zero}. We emphasize that without additional assumptions on the statistical model, it seems impossible to justify the convergence of this integral using the dominated convergence theorem. 

From \eqref{eq : def hat i c alpha} and with computations analogous to the one giving \eqref{eq : bound on i r x0}, we have
\begin{equation*}
	\hat{i}_{c,\alpha}(u_0)  \le \indi{(c,1-c)}(u_0) \frac{1}{c} \int_{-c/2+u_0}^{c/2+u_0} \frac{(\dot{p}_{\theta_0}(\Xi^{-1}(w)))^2}
	{{\modar p}_{\theta_0}(\Xi^{-1}(w))\nu_0(\Xi^{-1}(w))} dw.
\end{equation*}
Let us define the three following functions, {\modar for $u\in(0,1)$} :
\begin{gather*}
	k(u):= \frac{(\dot{p}_{\theta_0}(\Xi^{-1}(u)))^2}
	{{\modar p}_{\theta_0}(\Xi^{-1}(u))\nu_0(\Xi^{-1}(u))},
	\\ 
	M_c(u):=c^{-1}\int_{(u-c/2)\vee 0}^{(u+c/2) \wedge 1} k(w)dw,
	\quad
	M_*(u):=\sup_{0<c<1/2}	M_c(u)	.
\end{gather*}	
Then, we have $		\hat{i}_{c,\alpha}(u_0)   \le M_c(u_0) \le M_*(u_0)$. It is known from the Hardy--Littlewood maximal inequality that, for any $p>1$, $k\in L^p(0,1)$ implies $M_* \in L^p(0,1)$. However, the finiteness of the Fisher information means exactly that $k \in L^1(0,1)$ and the Hardy--Littlewood inequality in $L^p$-norm is not true for $p=1$. Consequently,  we can not deduce that $M_* \in L^1(0,1)$, which would be sufficient to dominate  the family 	$(\hat{i}_{c,\alpha})_{c,\alpha}$ by a {\modar $L^1(0,1)$} function. However, we can use the weak form of the Hardy--Littlewood inequality (e.g. see Theorem 3.17 in \cite{follandRealAnalysisModern1999}) : for $s>0$ and denoting the Lebesgue measure by $\lambda$,
\begin{equation} \label{eq : Hardy-Littlewood}
	\lambda\{ u \in (0,1) \mid M_*(u) \ge s\} \le 3 \frac{\norm{k}_{L^1(0,1)}}{s}.
\end{equation}

To prove the uniform integrability of the family, we need upper bound the following quantity, for $\lambda >0$,
\begin{align*}
	\int_0^1 \hat{i}_{c,\alpha}(u_0) \indi{\{\hat{i}_{c,\alpha}(u_0) \ge \lambda\}} du_0
&	\le 
	\int_0^1 M_c(u_0) \indi{\{M_*(u_0) \ge \lambda\}} du_0,	\quad \text{ where we used $	\hat{i}_{c,\alpha}(u_0)   \le M_c(u_0) \le M_*(u_0)$,}
\\
& \le \int_0^1 c^{-1} \left( \int_{(u_0-c/2)\vee 0}^{(u_0+c/2) \wedge 1} k(w)dw  \right)  \indi{\{M_*(u_0) \ge \lambda\}}  du_0, 
\quad \text{ by the definition of $M_c$,}
\\
&\le 
\begin{multlined}[t]
	\int_0^1 c^{-1} \left( \int_{(u_0-c/2)\vee 0}^{(u_0+c/2) \wedge 1} k(w) \indi{\{k(w) \ge \lambda'\}} dw  \right)  \indi{\{M_*(u_0) \ge \lambda\}}  du_0
	\\+
	\int_0^1 c^{-1} \left( \int_{(u_0-c/2)\vee 0}^{(u_0+c/2) \wedge 1} \lambda'dw  \right)  \indi{\{M_*(u_0) \ge \lambda\}}  du_0, 
	\quad \text{ for any $\lambda'>0$.}
\end{multlined}
\end{align*}
Using Fubini-Tonelli theorem on the first integral, and the simple control $\left( \int_{(u_0-c/2)\vee 0}^{(u_0+c/2) \wedge 1} \lambda'dw  \right) \le c \lambda'$ on the second one, we get
\begin{align*}
		\int_0^1 \hat{i}_{c,\alpha}(u_0) 
		\indi{\{\hat{i}_{c,\alpha}(u_0) \ge \lambda\}} du_0
	&
	\begin{multlined}[t]
		\le \int_0^1  k(w)\indi{\{k(w) \ge \lambda'\}}  \left( c^{-1}\int_0^1  
	\indi{\{\abs{w-u_0}\le c/2\}} du_0 \right)dw
	\\+	
	\lambda' \int_0^1  \indi{\{M_*(u_0) \ge \lambda\}}  du_0 
	\end{multlined}
	\\
& \le  \int_0^1  k(w)\indi{\{k(w) \ge \lambda'\}} dw +
\lambda' \int_0^1  \indi{\{M_*(u_0) \ge \lambda\}}  du_0 .
\end{align*}
Now, using \eqref{eq : Hardy-Littlewood}, we deduce
\begin{equation*}
			\int_0^1 \hat{i}_{c,\alpha}(u_0) \indi{\{\hat{i}_{c,\alpha}(u_0) \ge \lambda\}} du_0 \le  \int_0^1  k(w)\indi{\{k(w) \ge \lambda'\}} dw+3\frac{\lambda'}{\lambda} \norm{k}_{L^1(0,1)}.
\end{equation*}
Since $k \in L^1(0,1)$, the right hand side of the above equation can be made arbitrarily small by choosing  consecutively $\lambda'$ and $\lambda$ sufficiently large. Then, the fact that the right hand side of this equation is independent of $(c,\alpha)$ proves that the family $  (\hat{i}_{c,\alpha}(u_0))_{\lambda,\alpha} $ is uniformly integrable on $(0,1)$.
\end{proof}

\bibliographystyle{plain}   
\bibliography{bib}






\end{document}